\documentclass[12pt]{amsart}
\usepackage{a4wide,enumerate,color,mathrsfs,graphicx}
\allowdisplaybreaks

\let\pa\partial  
\let\na\nabla  
\let\eps\varepsilon  
\newcommand{\N}{{\mathbb N}}  
\newcommand{\R}{{\mathbb R}} 

\newcommand{\diver}{\operatorname{div}}

\newtheorem{theorem}{Theorem}   
\newtheorem{lemma}[theorem]{Lemma}   
\newtheorem{proposition}[theorem]{Proposition}   
\newtheorem{remark}[theorem]{Remark}

\begin{document}  

\title[Entropy-dissipating semi-discrete Runge-Kutta schemes]{
Entropy-dissipating semi-discrete Runge-Kutta schemes for nonlinear
diffusion equations}

\author[A. J\"{u}ngel]{Ansgar J\"{u}ngel}
\address{Institute for Analysis and Scientific Computing, Vienna University of  
Technology, Wiedner Hauptstra\ss e 8--10, 1040 Wien, Austria}
\email{juengel@tuwien.ac.at}

\author[S. Schuchnigg]{Stefan Schuchnigg}
\address{Institute for Analysis and Scientific Computing, Vienna University of  
Technology, Wiedner Hauptstra\ss e 8--10, 1040 Wien, Austria}
\email{stefan.schuchnigg@tuwien.ac.at}

\date{\today}

\thanks{The authors acknowledge partial support from   
the Austrian Science Fund (FWF), grants P24304, P27352, and W1245}  

\begin{abstract}
Semi-discrete Runge-Kutta schemes for nonlinear diffusion equations of
parabolic type are analyzed. Conditions are determined under which the
schemes dissipate the discrete entropy locally. The dissipation property
is a consequence of the concavity of the difference of the entropies 
at two consecutive time steps. The concavity property 
is shown to be related to the Bakry-Emery approach
and the geodesic convexity of the entropy. The abstract conditions
are verified for quasilinear parabolic equations (including the porous-medium
equation), a linear diffusion system, and the fourth-order quantum diffusion 
equation. Numerical experiments for
various Runge-Kutta finite-difference discretizations of the one-dimensional 
porous-medium equation show that the entropy-dissipation property is in fact global.
\end{abstract}



\maketitle

%%%%%%%%%%%%%%%%%%%%%%%%%%%%%%%%%%%%%%%%%%%%%%%%%%%%%%%%%%%%%%%%%%%%%%%%%%%%%%%

\section{Introduction}\label{sec.intro}

Evolution equations often contain some structural information reflecting 
inherent physical properties such as positivity of solutions, conservation laws, 
and entropy dissipation. Numerical schemes should be designed
in such a way that these structural features are preserved on the discrete level 
in order to obtain accurate and stable algorithms.
In the last decades, concepts of structure-preserving schemes,
geometric integration, and compatible discretization have been
developed \cite{CMO11}, but much less is known about the preservation
of entropy dissipation and large-time asymptotics. 

Entropy-stable schemes were derived by Tadmor
already in the 1980s \cite{Tad87} in the context of conservation laws, thus
without (physical) diffusion. Later,
entropy-dissipative schemes were developed for (finite-volume)
discretizations of diffusion equations in \cite{Bes12,Fil08,GlGa09}.
In \cite{CaGu15}, a finite-volume scheme which preserves the gradient-flow 
structure and hence the entropy structure is proposed.
All these schemes are based on the implicit Euler method and are
of first order (in time) only.
Higher-order time schemes with entropy-dissipating properties are investigated 
in very few papers. A second-order predictor-corrector approximation
was suggested in \cite{LiYu14}, while higher-order semi-implicit Runge-Kutta (DIRK)
methods, together with a spatial fourth-order central finite-difference
discretization, were investigated in \cite{BFR15}. In \cite{BEJ14,JuMi15},
multistep time approximations were employed, but they can be 
at most of second order and they dissipate only one entropy and not all
functionals dissipated by the continuous equation.
In this paper, we remove these restrictions by investigating 
time-discrete Runge-Kutta schemes of order $p\ge 1$ for general 
diffusion equations.

We stress the fact that we are interested in the analysis of entropy-dissipating
schemes by ``translating'' properties for the continuous equation to the 
semi-discrete level, i.e., we study the stability of the schemes. 
However, we will not investigate convergence, stiffness, or computational 
issues here (see e.g.\ \cite{BFR15}).

More precisely, we consider time discretizations of the abstract Cauchy problem
\begin{equation}\label{1.eq}
  \pa_t u(t) + A[u(t)] = 0, \quad t>0, \quad u(0)=u^0,
\end{equation}
where $A:D(A)\to X'$ is a (differential) operator defined on $D(A)\subset X$
and $X$ is a  Banach space with dual $X'$. In this paper, we restrict ourselves
to diffusion operators $A[u]$ defined on some Sobolev space with solutions
$u:\Omega\times(0,\infty)\to\R^n$, which may be vector-valued.
A typical example is $A[u]=\diver(a(u)\na u)$ defined on $X=L^2(\Omega)$
with domain $D(A)=H^2(\Omega)$, where $a:\R\to\R$ is a
smooth function  (see section \ref{sec.de}). 
Equation \eqref{1.eq} often possesses a Lyapunov functional
$H[u]=\int_\Omega h(u)dx$ (here called {\em entropy}), where $h:\R^n\to\R$, 
such that
$$
  \frac{dH}{dt}[u] = \int_\Omega h'(u)\pa_t u dx
	= -\int_\Omega h'(u)A[u] dx \le 0,
$$
at least when the {\em entropy production} $\int_\Omega h'(u)A[u] dx$ is nonnegative,
Here, $h'$ is the derivative of $h$ and 
$h'(u)A[u]$ is interpreted as the inner product of $h'(u)$ and $A[u]$
in $\R^n$. Furthermore,
if $h$ is convex, the convex Sobolev inequality $\int_\Omega h'(u)A[u]dx
\ge \kappa H[u]$ for some $\kappa>0$ may hold \cite{CJS15}, which implies 
that $dH/dt\le-\kappa H$ and hence
exponential convergence of $H[u]$ to zero with rate $\kappa$. The aim is
to design a higher-order time-discrete scheme which preserves this
entropy-dissipation property.

To this end, we propose the following semi-discrete Runge-Kutta
approximation of \eqref{1.eq}: Given $u^{k-1}\in X$, define
\begin{equation}\label{1.rk}
  u^{k} = u^{k-1} + \tau\sum_{i=1}^s b_iK_i, \quad
	K_i = -A\bigg[u^{k-1} + \tau\sum_{j=1}^s a_{ij}K_j\bigg], \quad i=1,\ldots,s,
\end{equation}
where $t^{k}$ are the time steps, 
$\tau=t^{k}-t^{k-1}>0$ is the uniform time step size, $u^k$ approximates $u(t^k)$,
and $s\ge 1$ denotes the number of Runge-Kutta stages. Since the Cauchy problem
is autonomous, the knots $c_1,\ldots,c_s$ are not needed here.
In concrete examples (see below), $u^k$ are functions from $\Omega$ to $\R^n$.
If $a_{ij}=0$ for $j\ge i$, the Runge-Kutta scheme
is explicit, otherwise it is implicit and a nonlinear system
of size $s$ has to be solved to compute $K_i$. 
We assume that scheme \eqref{1.rk} is solvable for $u^k:\Omega\to\R^n$.
 
Given $h:\R^n\to\R$, we wish to determine conditions 
under which the functional
\begin{equation}\label{1.H}
  H[u^k] = \int_\Omega h(u^k(x))dx
\end{equation}
is dissipated by the numerical scheme \eqref{1.rk},
\begin{equation}\label{1.ed}
  H[u^k] + \tau\int_\Omega A[u^k]h'(u^k)dx \le H[u^{k-1}], \quad k\in\N.
\end{equation}
In many examples (see below), $\int_\Omega A[u^k]h'(u^k)dx\ge 0$ 
and thus, the function $k\mapsto H[u^k]$ is decreasing. 
Such a property is the first step in proving
the preservation of the large-time asymptotics of the numerical scheme
(see Remark \ref{rem.exp}). 

Our main results, stated on an informal level, are as follows:
\begin{enumerate}[(i)]
\item We determine an abstract condition under which the discrete entropy-dissipation 
inequality \eqref{1.ed} holds for sufficiently small $\tau^k>0$.
This condition is made explicit for special
choices of $A$ and $h$, yielding entropy-dissipative implicit or explicit 
Runge-Kutta schemes of any order.
\item Numerical experiments for the porous-medium equation indicate that
$\tau^k$ may be chosen independent of the time step $k$, thus yielding
discrete entropy dissipation for all discrete times.
\item We show that for Runge-Kutta schemes of order $p\ge 2$,
the abstract condition in (i) is exactly the criterion of Liero and Mielke
\cite{LiMi13} to conclude geodesic 0-convexity of the entropy. In particular,
it is related to the Bakry-Emery condition \cite{BaEm85}.
\end{enumerate}

Let us describe the main results in more detail. 
We recall that the Runge-Kutta scheme \eqref{1.rk} is consistent if 
$\sum_{j=1}^s a_{ij}=c_i$ and $\sum_{i=1}^s b_i=1$. Furthermore, if 
$\sum_{i=1}^s b_ic_i=\frac12$, it is at least of order two \cite[Chap.~II]{HNW93}.
We introduce the number 
\begin{equation}\label{1.CRK}
  C_{\rm RK} = 2\sum_{i=1}^sb_i(1-c_i),
\end{equation}
which takes only three values:
\begin{align*}
  C_{\rm RK} &= 0\quad\mbox{for the implicit Euler scheme}, \\
	C_{\rm RK} &= 1\quad\mbox{for any Runge-Kutta scheme of order }p\ge 2, \\
	C_{\rm RK} &= 2\quad\mbox{for the explicit Euler scheme}.
\end{align*}

The {\em first main result} is an abstract entropy-dissipation property of scheme
\eqref{1.rk} for entropies of type \eqref{1.H}.

\begin{theorem}[Entropy-dissipation structure I]\label{thm.ed}
Let $h\in C^2(\R^n)$, let $A:D(A)\to X'$ be Fr\'echet differentiable
with Fr\'echet derivative $DA[u]:X\to X'$ at $u\in D(A)$, and let
$(u^k)$ be the Runge-Kutta solution to \eqref{1.rk}. Suppose that
\begin{equation}\label{1.I0}
  I_0^k :=
	\int_\Omega \big(C_{\rm RK}h'(u^k)DA[u^k](A[u^k])+h''(u^k)(A[u^k])^2\big)dx > 0.
\end{equation}
Then there exists $\tau^k>0$ such that for all $0<\tau\le\tau^k$,
\begin{equation}\label{1.edi}
  H[u^k] + \tau\int_\Omega A[u^k]h'(u^k)dx \le H[u^{k-1}].
\end{equation}
\end{theorem}

We assume that the solutions to \eqref{1.rk} are sufficiently regular
such that the integral \eqref{1.I0} can be defined. In the vector-valued case,
$h''(u^k)$ is the Hessian matrix and we interpret
$h''(u^k)(A[u^k])^2$ as the product $A[u^k]^\top h''(u^k)A[u^k]$.
For Runge-Kutta schemes of order $p\ge 2$ (for which $C_{\rm RK}=1$), 
the integral \eqref{1.I0}
corresponds exactly to the second-order time derivative of $H[u(t)]$ for solutions
$u(t)$ to the {\em continuous} equation \eqref{1.eq}.
Observe that the entropy-dissipation estimate \eqref{1.edi} is only {\em local}, 
since the time step restriction depends on the time step $k$.
For implicit Euler schemes (and convex entropies $h$), it is known that
$\tau^k$ can be chosen independent of $k$. For general Runge-Kutta
methods, we cannot prove rigorously that $\tau^k$ stays bounded from below
as $k\to\infty$. However, our numerical experiments in section \ref{sec.num}
indicate that inequality \eqref{1.edi} holds for sufficiently small $\tau>0$
uniformly in $k$. 

\begin{remark}[Exponential decay of the discrete entropy]\label{rem.exp}\rm 
If the convex Sobolev inequality $\int_\Omega A[u^k]h'(u^k)dx\ge
\kappa H[u^k]$ holds for some constant $\kappa>0$ 
and if there exists $\tau^*>0$ such that $\tau^k\ge\tau^*>0$
for all $k\in\N$, we infer from \eqref{1.edi} that
for $\tau:=\tau^*$,
$$
  H[u^k] \le (1+\kappa\tau)^{-k}H[u^0] = \exp(-\eta\kappa t^k)H[u^0],
	\quad\mbox{where }\eta = \frac{\log(1+\kappa\tau)}{\kappa\tau}
	< 1,
$$
which implies exponential decay of the discrete entropy with rate $\eta\kappa$.
This rate converges to the continuous rate $\kappa$ as $\tau\to 0$ and therefore,
it is asymptotically sharp.
\qed
\end{remark}

Theorem \ref{thm.ed} can be generalized to a larger class of entropies,
namely to so-called {\em first-order entropies}
\begin{equation}\label{1.F}
  F[u^k] = \int_\Omega |\na f(u^k)|^2 dx,
\end{equation}
where, for simplicity, we consider only the scalar case with $f:\R\to\R$.
An important example is the Fisher information with $f(u)=\sqrt{u}$.

\begin{theorem}[Entropy-dissipating structure II]\label{thm.ed2}
Let $f\in C^2(\R)$, let $A:D(A)\to X'$ be Fr\'echet differentiable, and
let $(u^k)$ be the Runge-Kutta solution to \eqref{1.rk}.
Assume that the boundary condition $\na f(u^k)\cdot\nu=0$ on $\pa\Omega$
is satisfied. Furthermore, suppose that
\begin{equation}\label{1.I1}
\begin{aligned}
  I_1^k :=
	\int_\Omega \Big( & |\na(f'(u^k)A[u^k]|^2
	- C_{\rm RK}\Delta f(u^k)f'(u^k)DA[u^k](A[u^k]) \\
	&{}- \Delta f(u^k)f''(u^k)(A[u^k])^2\Big)dx > 0.
\end{aligned}
\end{equation}
Then there exists $\tau^k>0$ such that for all $0<\tau\le\tau^k$,
$$
  F[u^k] + \tau\int_\Omega A[u^k]f'(u^k)dx \le F[u^{k-1}].
$$
\end{theorem}

The key idea of the proof of Theorem \ref{thm.ed} (and similarly for Theorem 
\ref{thm.ed2}) is a concavity property of the difference of the entropies 
at two consecutive time steps with respect to the time step size $\tau$. 
To explain this idea, let $u:=u^k$ be fixed and introduce $v(\tau):=u^{k-1}$.
Clearly, $v(0)=u$. A formal Taylor expansion of $G(\tau):=H[u]-H[v(\tau)]$ 
yields
$$
  H[u^k]-H[u^{k-1}] = G(\tau) = G(0) + \tau G'(0) + \frac{\tau^2}{2}G''(\xi^k),
$$
where $0<\xi^k<\tau$.
A computation, made explicit in section \ref{sec.meth}, shows that 
$G'(0)=\int_\Omega A[u^k]h'(u^k)dx$ and $G''(0)=-I_0^k$.
Now, if $G''(0)<0$, there exists $\tau^k>0$ such that $G''(\tau)\le 0$ for
$\tau\in [0,\tau^k]$ and in particular $G''(\xi^k)\le 0$. Consequently,
$G(\tau)\le \tau G'(0)$, which equals \eqref{1.ed}.  
The definition of $v(\tau)$ assumes implicitly that \eqref{1.rk} is 
{\em backward} solvable. We prove in Proposition \ref{prop.back} below that 
this property holds if the operator $A$ is a smooth self-mapping on $X$.

\begin{remark}[Discussion of $\tau^k$]\label{rem.tauk}\rm
Since $(u^k)$ is expected to converge to the stationary solution,
$\lim_{k\to\infty}I_0^k=0$. Thus, in principle, for larger values of $k$,
we expect that $\tau^k$ becomes smaller and smaller, thus restricting the
choice of time step sizes $\tau$. However, practically, the situation is better.
For instance, for the implicit Euler scheme, if $h$ is convex, we obtain
$$
  H[u^k] - H[u^{k-1}] \le \int_\Omega h'(u^k)(u^k-u^{k-1})dx
	= -\tau\int_\Omega h'(u^k)A[u^k]dx
$$
for {\em any} value of $\tau>0$. Moreover, for other (higher-order) 
Runge-Kutta schemes, the numerical experiments in section \ref{sec.num} indicate 
that there exists $\tau^*>0$ such that 
$G''(\tau)\le 0$ holds for all $\tau\in[0,\tau^*]$
uniformly in $k\in\N$. In this situation, inequality \eqref{1.edi} holds
for all $0<\tau\le\tau^*$, and thus our estimate is global. 
In fact, the function $G''$ is numerically even nonincreasing 
in some interval $[0,\tau^*]$ but we are not able to prove this analytically.
\qed
\end{remark}

The {\em second main result} is the specification of the abstract 
conditions \eqref{1.I0} and \eqref{1.I1} for a number of examples: 
a quasilinear diffusion equation,
porous-medium or fast-diffusion equations, a linear diffusion system, and
the fourth-order Derrida-Lebowitz-Speer-Spohn equation (see sections
\ref{sec.de}-\ref{sec.dlss} for details). 
For instance, for the porous-medium equation 
$$
  \pa_t u = \Delta(u^\beta) \mbox{ in }\Omega,\ t>0, \quad
	\na u^\beta\cdot\nu = 0 \mbox{ on }\pa\Omega, \quad u(0)=u^0,
$$
we show that the Runge-Kutta scheme scheme satisfies
$$
  H[u^k] + \tau\beta\int_\Omega(u^k)^{\alpha+\beta-2}|\na u^k|^2 dx
	\le H[u^{k-1}], \quad\mbox{where }
	H[u]=\frac{1}{\alpha(\alpha+1)}\int_\Omega u^{\alpha+1} dx,
$$
for $0<\tau\le\tau^k$ and all $(\alpha,\beta)$ belonging to some region
in $[0,\infty)^2$ (see Figure \ref{fig.0th} below). 
For $\alpha=0$, we write $H[u]=\int_\Omega u(\log u-1)dx$.
In one space dimension and for Runge-Kutta schemes of order $p\ge 2$, 
this region becomes $-2<\alpha-\beta<1$, which is the same condition as for the 
continuous equation (except the boundary values). 
Furthermore, the first-order entropy \eqref{1.F} is dissipated for 
Runge-Kutta schemes of order $p\ge 2$, in one space dimension,
$$
  F[u^k] + \tau C_{\alpha,\beta}\int_\Omega (u^k)^{\alpha+\beta-2}(u^k)_{xx}^2 dx
	\le F[u^{k-1}], \quad\mbox{where }
	F[u]=\int_\Omega (u^{\alpha/2})_x^2 dx,
$$
for $0<\tau\le\tau^k$ and all $(\alpha,\beta)$ belonging to the region
shown in Figure \ref{fig.1st} below, and $C_{\alpha,\beta}>0$ is some constant.
This region is smaller than the region of admissible values $(\alpha,\beta)$
for the continuous entropy. The borders of that region are indicated in the
figure by dashed lines.

The proof of the above results, and namely of $G''(0)<0$, is based on
systematic integration by parts \cite{JuMa06}. The idea of the method
is to formulate integration by parts as manipulations with polynomials
and to conclude the inequality $G''(0)<0$ from a polynomial decision problem.
This problem can be solved directly or by using computer algebra software.

Our {\em third main result} is the relation to geodesic 0-convexity of the
entropy and the Bakry-Emery approach when 
$C_{\rm RK}=1$ (Runge-Kutta scheme of order $p\ge 2$).
Liero and Mielke formulate in \cite{LiMi13} the abstract Cauchy problem 
\eqref{1.eq} as the gradient flow
$$
  \pa_t u = -K[u]DH[u], \quad t>0, \quad u(0)=u^0,
$$
where the Onsager operator $K[u]$ describes the sum of diffusion and reaction
terms. For instance, if $A[u]=\diver(a(u)\na u)$, we can write
$A[u]=\diver(a(u)h''(u)^{-1}\na h'(u))$ and thus, identifying 
$h'(u)$ and $DH[u]$, we have $K[u]\xi=\diver(a(u)h''(u)^{-1}\na\xi)$.
It is shown in \cite{LiMi13} that the entropy $H$ is geodesic $\lambda$-convex if
the inequality
\begin{equation}\label{1.M}
  M(u,\xi) := \langle \xi,DA[u]K[u]\xi\rangle - \frac12\langle\xi,DK[u]A[u]\xi\rangle
	\ge \lambda\langle\xi,K[u]\xi\rangle
\end{equation}
holds for all suitable $u$ and $\xi$. We will prove in section \ref{sec.meth} that
$$
  G''(0) = 2M(u^k,h'(u^k)).
$$
Hence, if $G''(0)\le 0$ then \eqref{1.M} with $\lambda=0$
is satisfied for $u=u^k$ and $\xi=h'(u^k)$,
yielding geodesic $0$-convexity for the semi-discrete entropy.
Moreover, if $G''(0)\le -\lambda G'(0)$ then we obtain geodesic $\lambda$-convexity.
Since $G'(0)=-dH[u]/dt$ and $G''(0)=-d^2H[u]/dt^2$ in the continuous setting, 
the inequality $G''(0)\le -\lambda G'(0)$ can be written as
$$
  \frac{d^2H}{dt^2}[u] \ge -\lambda \frac{dH}{dt}[u],
$$
which corresponds to a variant of the Bakry-Emery condition \cite{BaEm85}, yielding
exponential convergence of $H[u]$ (if $\tau^k\ge\tau^*>0$ for all $k$).
Thus, our results constitute a first step towards a 
{\em discrete Bakry-Emery approach}.

The paper is organized as follows. The abstract method, i.e.\ the proof of
backward solvability and of Theorems \ref{thm.ed} and \ref{thm.ed2}, is
presented in section \ref{sec.meth}. The method is applied in the subsequent
sections to a scalar diffusion equation (section \ref{sec.de}), the porous-medium
equation (section \ref{sec.pme}), a linear diffusion system 
(section \ref{sec.sys}), and the fourth-order Derrida-Lebowitz-Speer-Spohn
equation (section \ref{sec.dlss}). Finally, section \ref{sec.num} is devoted
to some numerical experiments showing that $G''$ is negative in some
interval $[0,\tau^*]$.

%%%%%%%%%%%%%%%%%%%%%%%%%%%%%%%%%%%%%%%%%%%%%%%%%%%%%%%%%%%%%%%%%%%%%%%%%%%%%%

\section{The abstract method}\label{sec.meth}

In this section, we show that the Runge-Kutta scheme is backward solvable if
$A$ is a self-mapping and we prove Theorems \ref{thm.ed} and \ref{thm.ed2}.

\begin{proposition}[Backward solvability]\label{prop.back}
Let $(\tau,u^{k})\in[0,\infty)\times X$, where $X$ is some Banach space,
and let $A\in C^2(X,X)$ be a self-mapping. Then there exists $\tau_0>0$,
a neighborhood $V\subset X$ of $u^{k}$, and a function
$v\in C^2([0,\tau_0);X)$ such that \eqref{1.rk} holds for $u^{k-1}:=v(\tau)$.
Moreover, 
\begin{equation}\label{v0}
   v(0)=0, \quad v'(0)=A[u], \quad\mbox{and}\quad v''(0)=C_{\rm RK}DA[u](A[u]).
\end{equation}
\end{proposition}

The self-mapping assumption is strong for differential operators $A$ but
it is somehow natural in the context of Runge-Kutta methods and valid for
smooth solutions.

\begin{proof}
The idea of the proof is to apply the implicit function theorem in Banach
spaces (see \cite[Corollary 15.1]{Dei85}). To this end, we set $u:=u^k$ and define
the mapping $J=(J_0,\ldots,J_{s}):\R\times X^{s+1}\to X^{s+1}$ by
\begin{align*}
  J_0(\tau,y) &= v - u + \tau\sum_{i=1}^s b_i k_i, \quad\mbox{where } 
	y=(k_1,\ldots,k_s,v), \\
	J_i(\tau,y) &= k_i + A\bigg[v + \tau\sum_{j=1}^s a_{ij}k_j\bigg],
	\quad i=1,\ldots,s.
\end{align*}
The Fr\'echet derivative of $J$ in the direction of $(\tau_h,y_h)$,
where $y_h=(k_{h1},\ldots,k_{hs},v_h)$, reads as
\begin{align*}
  DJ_0(\tau,y)(\tau_h,y_h) &= v_h + \tau_h\sum_{i=1}^s b_i k_{i}
	+ \tau\sum_{i=1}^s b_i k_{hi}, \\
	DJ_i(\tau,y)(\tau_h,y_h) &= k_{hi} 
	+ DA\bigg[v + \tau\sum_{j=1}^s a_{ij}k_j\bigg]
	\bigg(v_h + \tau_h\sum_{j=1}^s a_{ij}k_{j} + \tau\sum_{j=1}^s a_{ij}k_{hj}\bigg), 
\end{align*}
where $i=1,\ldots,s$. Let $\tau_0=0$ and $y_0=(-A[u],\ldots,-A[u],u)$. Then
$J(\tau_0,y_0)=0$ and
$$
  DJ_0(\tau_0,y_0)(0,y_h) = v_h, \quad 
	DJ_i(\tau_0,y_0)(0,y_h) = k_{ih} + DA[u](v_h), \quad i=1,\ldots,s.
$$
The mapping $y_h\mapsto DJ(\tau_0,y_0)(0,y_h)$ is clearly an isomorphism
from $X^{s+1}$ onto $X^{s+1}$. 
By the implicit function theorem, there exist an interval $U\subset [0,\tau_0)$,
a neighborhood $V\subset X^{s+1}$ of $y_0$, and a function $(k,v)\in 
C^2([0,\tau_0);V)$ such that $(k,v)(0)=(-A[u],\ldots,-A[u],u)$ and
$J(\tau,k(\tau),v(\tau))=0$ for all $\tau\in[0,\tau_0)$. 

Implicit differentiation of $J(\tau,k(\tau),v(\tau))=0$ yields
\begin{align*}
  0 &= v'(\tau) + \sum_{i=1}^s b_i k_i(\tau) + \tau\sum_{i=1}^s b_i k_i'(\tau), \\
	0 &= k_i'(\tau) + DA\bigg[v + \tau\sum_{j=1}^s a_{ij}k_j(\tau)\bigg]
	\bigg(v'(\tau) + \sum_{j=1}^s a_{ij} k_j(\tau) + \tau\sum_{j=1}^s a_{ij} 
	k_j'(\tau)\bigg),
\end{align*}
where $i=1,\ldots,s$ and $\tau\in[0,\tau_0)$. Using $\sum_{i=1}^s b_i=1$
and $\sum_{j=1}^s a_{ij}=c_i$, we infer that
\begin{align}
  v'(0) &= -\sum_{i=1}^s b_i k_i(0) = \sum_{i=1}^s b_iA[u] = A[u], \nonumber \\
	k_i'(0) &= -DA[u]\bigg(A[u] - \sum_{j=1}^s a_{ij}A[u]\bigg)
	= -(1-c_i)DA[u](A[u]). \label{aux1}
\end{align} 
Differentiating $J_0(\tau,k(\tau),v(\tau))=0$ twice leads to
$$
  0 = v''(\tau) + 2\sum_{i=1}^s b_i k_i'(\tau) + \tau\sum_{i=1}^s b_ik_i''(\tau).
$$
Because of \eqref{aux1}, this reads at $\tau=0$ as
$$
  v''(0) = -2\sum_{i=1}^s b_ik_i'(0) = 2\sum_{i=1}^s b_i(1-c_i)DA[u](A[u])
	= C_{\rm RK}DA[u](A[u]).
$$
This finishes the proof.
\end{proof}

We prove now Theorems \ref{thm.ed} and \ref{thm.ed2}.

\begin{proof}[Proof of Theorem \ref{thm.ed}.]
We set $u:=u^k$. By Proposition \ref{prop.back},
there exists a backward solution $v\in C^2([0,\tau_0))$ such that
$v(0)=u$, $v'(0)=A[u]$, and $v''(0)=C_{\rm RK}DA[u](A[u])$. Furthermore,
the function $G(\tau)=\int_\Omega(h(u)-h(v(\tau)))dx$ satisfies $G(0)=0$,
\begin{align*}
  G'(0) &= -\int_\Omega h'(v(0))v'(0)dx = -\int_\Omega h'(u)A[u]dx, \\
	G''(0) &= -\int_\Omega\big(h'(v(0))v''(0)
	+ h''(v(0))v'(0)^2\big)dx \\
	&= -\int_\Omega\big(C_{\rm RK}h'(u)DA[u](A[u]) + h''(u)(A[u])^2\big)dx
	= -I_0^k < 0,
\end{align*}
using the assumption.
By continuity, there exists $0<\tau^k<\tau_0$ such that $G''(\xi)\le 0$ for
$0\le\xi\le\tau^k$. 
Then the Taylor expansion $G(\tau)=G(0)+G'(0)\tau+\frac12G''(\xi)\tau^2
\le G'(0)\tau$ concludes the proof.
\end{proof}

\begin{proof}[Proof of Theorem \ref{thm.ed2}]
Following the lines of the previous proof, it is sufficient to compute
$G'(0)$ and $G''(0)$, where now 
$G(\tau)=\int_\Omega(|\na f(u)|^2-|\na f(v(\tau))|^2)dx$. 
Using  integration by parts and the boundary condition $\na f(v)\cdot\nu=0$ 
on $\pa\Omega$, we compute
$$
  G'(0) = -\int_\Omega\na f(v(0))\cdot\na\big( f'(v(0))v'(0)\big)dx
	= \int_\Omega \Delta f(u)f'(v(\tau))A[u]dx,
$$
since $v(0)=u$ and $v'(0)=A[u]$. Furthermore,
again integrating by parts,
\begin{align*}
  G''(\tau) &= -\int_\Omega\Big(\big|\na\big(f'(v(\tau))v'(\tau)\big)\big|^2
	+ \na f(v(\tau))\cdot\na\big(f''(v(\tau))(v'(\tau))^2\big) \\
	&\phantom{xx}{}+ \na f(v(\tau))\cdot\na\big(f'(v(\tau))v''(\tau)\big)\Big)dx \\
	&= -\int_\Omega\Big(\big|\na\big(f'(v(\tau))v'(\tau)\big)\big|^2
	- \Delta f(v(\tau))f''(v(\tau))(v'(\tau))^2 \\
	&\phantom{xx}{}- \Delta f(v(\tau))f'(v(\tau))v''(\tau)\Big)dx.
\end{align*}
Since $v''(0)=C_{\rm RK}DA[u](A[u])$, this reduces at $\tau=0$ to
\begin{align*}
  G''(0) &= -\int_\Omega\Big(|\na(f'(u)A[u])|^2 - \Delta f(u)f''(u)(A[u])^2
	- C_{\rm RK}\Delta f(u)f'(u)DA[u](A[u])\Big)dx.
\end{align*}
This expression equals $-I_1^k$, and the result follows.
\end{proof}

Finally, we show that $G''(0)$ for entropies \eqref{1.H} is related to the 
geodesic convexity condition of \cite{LiMi13}.

\begin{lemma}
Let $A[u]=K(u)DH[u]$ for some symmetric operator $K:D(A)\to X$ 
and Fr\'echet derivative
$DH[u]$, let $G$ be defined as in the proof of Theorem \ref{thm.ed} for a 
solution $u^k$ to the Runge-Kutta scheme \eqref{1.rk} of order $p\ge 2$, and
let $M(u,\xi)$ be given by \eqref{1.M}. Then
$$
  G''(0) = -2M(u^k,DH[u^k]).
$$
\end{lemma}

\begin{proof}
The proof is just a (formal) calculation. Recall that for Runge-Kutta schemes
of order $p\ge 2$, we have $C_{\rm RK}=1$.
Set $u:=u^k$ and identify $DH[u]$ with $\xi=h'(u)$. 
Inserting the expression $DA[u](v)=DK[u](v)h'(u)+ K[u]h''(u)v$ into the
definition of $G''(0)$, we find that
\begin{align*}
  -G''(0) &= \langle\xi,DA[u](A[u])\rangle + \langle A[u],h''(u)A[u]\rangle \\
	&= \big\langle\xi,DK[u](A[u])\xi + K[u]h''(u)A[u]\big\rangle
	+ \langle A[u],h''(u)A[u]\rangle \\
	&= \langle\xi,DK[u](K[u]\xi)\xi\rangle + \langle\xi,K[u]h''(u)K[u]\xi\rangle
	+ \langle K[u]\xi,h''(u)K[u]\xi\rangle \\
	&= \langle\xi,DK[u](K[u]\xi)\xi\rangle + 2\langle\xi,K[u]h''(u)K[u]\xi\rangle,
\end{align*}
since $K[u]$ is assumed to be symmetric. Rearranging the terms, we obtain
\begin{align*}
  -G''(0) &= 2\langle\xi,DK[u](K[u]\xi)\xi\rangle 
	+ 2\langle\xi,K[u]h''(u)K[u]\xi\rangle 
	- \langle\xi,DK[u](K[u]\xi)\rangle \\
	&= 2\langle\xi,DA[u](K[u]\xi)\xi\rangle - \langle\xi,DK[u](A[u])\rangle
	= 2M(u,\xi),
\end{align*}
which proves the claim.
\end{proof}

%%%%%%%%%%%%%%%%%%%%%%%%%%%%%%%%%%%%%%%%%%%%%%%%%%%%%%%%%%%%%%%%%%%%%%%%%%%%%%

\section{Scalar diffusion equation}\label{sec.de}

In this section, we analyze time-discrete Runge-Kutta schemes of the
diffusion equation
\begin{equation}\label{de.eq}
  \pa_t u = \diver(a(u)\na u), \quad t>0, \quad u(0)=u^0,
\end{equation}
with periodic or homogeneous Neumann boundary conditions.
This equation, also including a drift term, was analyzed in \cite{LiMi13}
in the context of geodesic convexity. 
Our results are similar to
those in \cite{LiMi13} but we consider the time-discrete and not the continuous
equation and we employ systematic integration by parts \cite{JuMa06}.

Setting $\mu(u)=a(u)/h''(u)$, we can write the diffusion equation as a formal
gradient flow:
$$
  \pa_t u = -A[u] := \diver(\mu(u)\na h'(u)), \quad t>0.
$$
We prove that the Runge-Kutta scheme \eqref{1.rk} dissipates all convex entropies
subject to some conditions on the functions $\mu$ and $h$.

\begin{theorem}\label{thm.de}
Let $\Omega\subset\R^d$ be convex with smooth boundary.
Let $(u^k)$ be a sequence of (smooth) solutions to the Runge-Kutta scheme 
\eqref{1.rk} of the diffusion equation \eqref{de.eq}. Let $k\in\N$ be fixed and
$u^k$ be not equal to the constant steady state of \eqref{de.eq}.
We suppose that for all admissible $u$, it holds that
$a(u)\ge 0$, $h''(u)\ge 0$,
\begin{align}\label{de.cond1}
  & b(u) := \frac23(C_{\rm RK}+1)\int_{u_0}^u \mu(v)\mu'(v)h''(v)dv \ge 0, \\
	& \frac{d-1}{d}b(u) \le (C_{\rm RK}+1)h''(u)\mu(u)^2, \label{de.cond2} \\
	& (C_{\rm RK}+2)\mu(u)\mu''(u) + (C_{\rm RK}-1)\mu'(u)^2 < 0. \label{de.cond3}
\end{align}
Then there exists $\tau^k>0$ such that for all $0<\tau<\tau^k$,
$$
  H[u^k] + \tau\int_\Omega h''(u^k)a(u^k)|\na u^k|^2 dx \le H[u^{k-1}].
$$
\end{theorem}

Conditions \eqref{de.cond1}-\eqref{de.cond2} correspond to (4.12) in \cite{LiMi13}.
Condition \eqref{de.cond3} is satisfied for concave functions $\mu$, 
except for the explicit Euler scheme 
($C_{\rm RK}=2$) for which we need additionally $4\mu\mu''+(\mu')^2<0$.
For the implicit Euler scheme, we may allow even for nonconcave mobilities $\mu$,
e.g.\ $\mu(u)=u^\gamma$ for $1<\gamma<2$.

\begin{proof}
According to Theorem \ref{thm.ed}, we only need to show that $I_0^k=-G''(0)>0$.
To simplify, we set $u:=u^k$.
First, we observe that the boundary condition $\na u\cdot\nu=0$ on $\Omega$
implies that $0=\pa_t\na u\cdot\nu=\na\pa_t u\cdot\nu=-\na A[u]\cdot\nu$
on $\pa\Omega$.
Using $DA[u](A[u]) = \diver(a'(u)A[u]\na u+a(u)\na A[u])=\Delta(a(u)A[u])$, 
the abbreviation $\xi=h'(u)$, and integration by parts, we compute
\begin{align*}
  G''(0) &= -\int_\Omega\Big(C_{\rm RK}h'(u)\Delta(a(u)A[u]) 
	+ h''(u)\big(\diver(\mu(u)\na h'(u))\big)^2\Big)dx \\
  &= \int_\Omega\Big(C_{\rm RK}\na h'(u)\cdot\na(a(u)A[u])
	- h''(u)\big(\mu'(u)\na u\cdot\na h'(u)+\mu(u)\Delta h'(u)\big)^2\Big)dx \\
	&= -\int_\Omega\bigg(C_{\rm RK}\Delta\xi a(u)A[u]
	+ h''(u)\left(\frac{\mu'(u)}{h''(u)}|\na\xi|^2 + \mu(u)\Delta\xi\right)^2
	\bigg)dx.
\end{align*}
The boundary integrals vanish since $\na u\cdot\nu=\na A[u]\cdot\nu=0$ on
$\pa\Omega$. Replacing $A[u]$ by $\diver(\mu(u)\na\xi)
=\mu(u)\Delta\xi+\mu'(u)|\na\xi|^2/h''(u)$ and expanding the
square, we arrive at
\begin{align}
  G''(0) &= -\int_\Omega\bigg(\big(C_{\rm RK}a(u)\mu(u)+h''(u)\mu(u)^2\big)
	(\Delta\xi)^2 \nonumber \\
	&\phantom{xx}{}+ \left(C_{\rm RK}a(u)\frac{\mu'(u)}{h''(u)}+2\mu(u)\mu'(u)\right) 
	\Delta\xi|\na\xi|^2 + \frac{\mu'(u)^2}{h''(u)}|\na\xi|^4\bigg)dx 
	\label{diffeq.G2} \\
	&= -\int_\Omega\big((C_{\rm RK}+1)h''(u)\mu(u)^2 \xi_L^2
	+ (C_{\rm RK}+2)\mu(u)\mu'(u)\xi_L\xi_G^2 + \mu'(u)^2 h''(u)^{-1}\xi_G^4\big)dx,
	\nonumber
\end{align}
where we have employed the identity $a(u)=\mu(u)h''(u)$ and the abbreviations 
$\xi_G=|\na\xi|$ and $\xi_L=\Delta\xi$. 

We apply now the method of systematic integration by parts \cite{JuMa06}.
The idea is to identify useful integration-by-parts formulas and to add
them to $G''(0)$ without changing the sign of $G''(0)$. The first formula is given by
\begin{equation}\label{de.ibp1}
  \int_\Omega\diver\big(\Gamma_1(u)(\na^2\xi-\Delta\xi{\mathbb I})\cdot\na\xi\big)dx
	= \int_{\pa\Omega}\Gamma_1(u)\na\xi^\top(\na^2\xi-\Delta\xi{\mathbb I})\nu ds,
\end{equation}
where $\Gamma_1(u)\le 0$ is an arbitrary (smooth) scalar function which still 
needs to be chosen, and $\mathbb{I}$ is the unit matrix in $\R^{d\times d}$.
The left-hand side can be expanded as
\begin{align*}
  \int_\Omega & \left(\frac{\Gamma_1'(u)}{h''(u)}\na\xi^\top
	(\na^2\xi-\Delta\xi{\mathbb I})
	\na\xi + \Gamma_1(u)\na^2\xi:(\na^2\xi-\Delta\xi{\mathbb I})\right)dx \\
  &= \int_\Omega\left(\frac{\Gamma_1(u)}{h''(u)}\xi_{GHG} 
	- \frac{\Gamma_1'(u)}{h''(u)}\xi_L\xi_G^2 
	+ \Gamma_1(u)\xi_H^2 - \Gamma_1(u)\xi_L^2\right)dx,
\end{align*}
where we have set $\xi_{GHG}=\na\xi^\top\na^2\xi\na\xi$ and $\xi_H=|\na^2\xi|$.
The boundary integral in \eqref{de.ibp1} becomes
$$
  \int_{\pa\Omega}\Gamma_1(u)\left(\frac12\na(|\na\xi|^2)
	-\Delta\xi\na\xi\right)\cdot\nu ds
	= \frac12\int_{\pa\Omega}\Gamma_1(u)\na(|\na\xi|^2)\cdot\nu ds \ge 0,
$$ 
since $\Gamma_1(u)\le 0$, $\na\xi\cdot\nu=0$
on $\pa\Omega$, and it holds that $\na(|\na\xi|^2)\cdot\nu\le 0$ on $\pa\Omega$ 
for all smooth functions satisfying $\na\xi\cdot\nu=0$ on $\pa\Omega$
\cite[Prop.~4.2]{LiMi13}. Here we need the convexity of $\Omega$.
Thus, the first integration-by-parts formula becomes
\begin{equation}\label{de.J1}
  J_1 := \int_\Omega\left(\frac{\Gamma_1'(u)}{h''(u)}\xi_{GHG} 
	- \frac{\Gamma_1'(u)}{h''(u)}\xi_L\xi_G^2 + \Gamma_1(u)\xi_H^2 
	- \Gamma_1(u)\xi_L^2\right)dx \ge 0.
\end{equation}
The second formula reads as
\begin{align}\label{de.J2}
  0 &= \int_\Omega\diver\big(\Gamma_2(u)|\na\xi|^2\na\xi)dx \\
	&= \int_\Omega\left(\frac{\Gamma_2'(u)}{h''(u)}\xi_G^4 + 2\Gamma_2(u)\xi_{GHG}
	+ \Gamma_2(u)\xi_L\xi_G^2\right)dx =: J_2, \nonumber
\end{align}
where $\Gamma_2$ is an arbitrary scalar function.
The goal is to find functions $\Gamma_1(u)\le 0$ and $\Gamma_2(u)$ such that 
$G''(0)\le G''(0)+J_1+J_2 < 0$.

According to \cite{JuMa08}, the computations simplify if we introduce the variables 
$\xi_R$ and $\xi_S$ satisfying
$$
  (d-1)\xi_G^2\xi_S = \xi_{GHG} - \frac{1}{d}\xi_L\xi_G^2, \quad
	\xi_H^2 = \frac{1}{d}\xi_L^2 + d(d-1)\xi_S^2 + \xi_R^2.
$$
The existence of $\xi_R$ follows from the inequality
$$
  \xi_H^2 = |\na^2\xi|^2 \ge \frac{1}{d}(\Delta\xi)^2 + \frac{d}{d-1}
	\left(\frac{\na\xi^\top\na^2\xi\na\xi}{\na\xi^2} - \frac{\Delta\xi}{d}\right)^2
	= \frac{1}{d}\xi_L^2 + d(d-1)\xi_S^2,
$$
which is proven in \cite[Lemma 2.1]{JuMa08}. Then
\begin{equation}\label{de.poly}
  G''(0) \le G''(0) + J_1 + J_2 
	= -\int_\Omega\big(a_1\xi_L^2 + a_2 \xi_L\xi_G^2 + a_3\xi_G^4
	+ a_4\xi_S\xi_G^2 + a_5\xi_R^2 + a_6\xi_S^2\big)dx,
\end{equation}
where 
\begin{equation}\label{de.ai}
\begin{aligned}
  a_1 &= (C_{\rm RK}+1)h''(u)\mu(u)^2 + \left(1-\frac{1}{d}\right)\Gamma_1(u), \\
	a_2 &= (C_{\rm RK}+2)\mu(u)\mu'(u) 
	+ \left(1-\frac{1}{d}\right)\frac{\Gamma_1'(u)}{h''(u)}
	- \left(\frac{2}{d}+1\right)\Gamma_2(u), \\
	a_3 &= \frac{\mu'(u)^2-\Gamma_2'(u)}{h''(u)}, \quad
	a_4 = -(d-1)\left(\frac{\Gamma_1'(u)}{h''(u)} + 2\Gamma_2(u)\right), \\
	a_5 &= -\Gamma_1(u), \quad a_6 = -d(d-1)\Gamma_1(u).
\end{aligned}
\end{equation}
The aim now is to determine conditions on $a_1,\ldots,a_6$ such that the
polynomial $P(\xi)=a_1\xi_L^2 + a_2 \xi_L\xi_G^2 + a_3\xi_G^4 + a_4\xi_S\xi_G^2 
+ a_5\xi_R^2 + a_6\xi_S^2$ is nonnegative
as this implies that $G''(0)\le 0$. In the general case, this
leads to nonlinear ordinary differential equations for $\Gamma_1$ and $\Gamma_2$
which cannot be easily solved. A possible solution is to require that the
coefficients of the mixed terms vanish, i.e.\ $a_2=a_4=0$, and that
the remaining coefficients are nonnegative. The case $d=1$ being simpler
than the general case (since $J_1$ is not necessary), we assume that $d>1$.
Then $a_4=0$ implies that $\Gamma_1'(u)/h''(u)=-2\Gamma_2(u)$. 
Replacing $\Gamma_1'(u)/h''(u)$ by $-2\Gamma_2(u)$
in $a_2=0$ gives
$$
  \Gamma_2(u) = \frac{C_{\rm RK}+2}{3}\mu(u)\mu'(u).
$$
On the other hand, replacing $\Gamma_2(u)$ by $-\Gamma_1'(u)/(2h''(u))$ in $a_2=0$,
we find that
$$
  \Gamma_1'(u) = -\frac23(C_{\rm RK}+2)\mu(u)\mu'(u)h''(u)
$$
or, after integration,
$$
  \Gamma_1(u) = -\frac23(C_{\rm RK}+2)\int_{u_0}^u\mu(v)\mu'(v)h''(v)dv.
$$
These functions have to satisfy the conditions
\begin{align*}
  a_1 &\ge 0\quad\mbox{or}\quad \frac{d-1}{d}\Gamma_1(u) \ge
	-(C_{\rm RK}+1)h''(u)\mu(u)^2, \\
  a_3 &\ge 0\quad\mbox{or}\quad 
	(C_{\rm RK}+2)\mu(u)\mu''(u) + (C_{\rm RK}-1)\mu'(u)^2 \le 0, \\
	a_5 &\ge 0\quad\mbox{or}\quad \Gamma_1(u)\le 0\quad\mbox{for all }u,
\end{align*}
Note that  $a_1\ge 0$ and $a_5\ge 0$ correspond to \eqref{de.cond2} and
\eqref{de.cond1}, respectively.
This shows that $P(\xi)\ge 0$ for all $\xi\in\R^4$ and $G''(0)\le 0$.

If $G''(0)=0$, the nonnegative polynomial $P$, which depends on $x\in\Omega$
via $\xi$, has to vanish. In particular, $a_3\xi_G^4=a_3|\na u|^4=0$ in
$\Omega$. As $a_3>0$ by assumption, $u(x)=\mbox{const.}$ for $x\in\Omega$.
This contradicts the hypothesis that $u$ is not a steady state.
Consequently, $G''(0)<0$, and we finish the proof by setting $b(u)=-\Gamma_1(u)$.
\end{proof}

%%%%%%%%%%%%%%%%%%%%%%%%%%%%%%%%%%%%%%%%%%%%%%%%%%%%%%%%%%%%%%%%%%%%%%%%%%%%%%

\section{Porous-medium equation}\label{sec.pme}

The results of the previous section can be applied in principle to the
Runge-Kutta scheme for the porous-medium or fast-diffusion equation
\begin{equation}\label{pm.eq}
  \pa_t u = \Delta(u^\beta)\quad\mbox{in }\Omega,\ t>0, \quad
	\na u^\beta\cdot\nu=0\quad\mbox{on }\pa\Omega, \quad u(0)=u^0,
\end{equation}
where $\beta>0$. 
It can be seen that conditions \eqref{de.cond1}-\eqref{de.cond3} are
not optimal for particular entropies. 
This is not surprising since we have neglected the mixed terms
in the polynomial in \eqref{de.poly} (i.e.\ $a_2=a_4=0$) which is not optimal.
In this section, we make a different approach by making an ansatz for the
functions $\Gamma_1$ and $\Gamma_2$, considering both zeroth-order and
first-order entropies.

%%%%%%%%%%%%%%%%

\subsection{Zeroth-order entropies}

We prove the following result.

\begin{theorem}\label{thm.pm}
Let $\Omega\subset\R^d$ be convex with smooth boundary.
Let $(u^k)$ be a sequence of (smooth) solutions to the Runge-Kutta scheme
\eqref{1.rk} for  \eqref{pm.eq}. Let the entropy be given by 
$H[u]=\alpha^{-1}(\alpha+1)^{-1}\int_\Omega u^{\alpha+1}dx$ with $\alpha>0$, 
let $k\in\N$, and let $u^k$ be not the constant steady state of \eqref{pm.eq}.
There exists a nonempty region $R_0(d)\subset(0,\infty)^2$ and $\tau^k>0$ 
such that for all $(\alpha,\beta)\in R_0(d)$ and $0<\tau\le\tau^k$,
$$
  H[u^k] + \tau\beta\int_\Omega (u^k)^{\alpha+\beta-2}|\na u^k|^2 dx
	\le H[u^{k-1}], \quad k\in\N.
$$
In one space dimension, we have 
\begin{align*}
  &\mbox{implicit Euler:} & R_0(1) &= (0,\infty)^2, \\
	&\mbox{Runge-Kutta of order }p\ge 2: &
	R_0(1) &= \big\{(\alpha,\beta)\in(0,\infty)^2: -2<\alpha-\beta<1\}, \\
	&\mbox{explicit Euler:} & 
  R_0(1) &= \big\{(\alpha,\beta)\in(0,\infty)^2: -1<\alpha-\beta<1\}.
\end{align*}
\end{theorem}

For the implicit Euler scheme, the theorem shows that 
any positive values for $(\alpha,\beta)$ is admissible
which corresponds to the continuous situation. For the Runge-Kutta case
with $C_{\rm RK}=1$, our condition is more restrictive. As expected, the
explicit Euler scheme requires the most restrictive condition.
The set $R_0(d)$ is illustrated in Figure \ref{fig.0th} for $d=2$ and $d=10$.

\begin{figure}[ht]
\includegraphics[width=65mm,height=55mm]{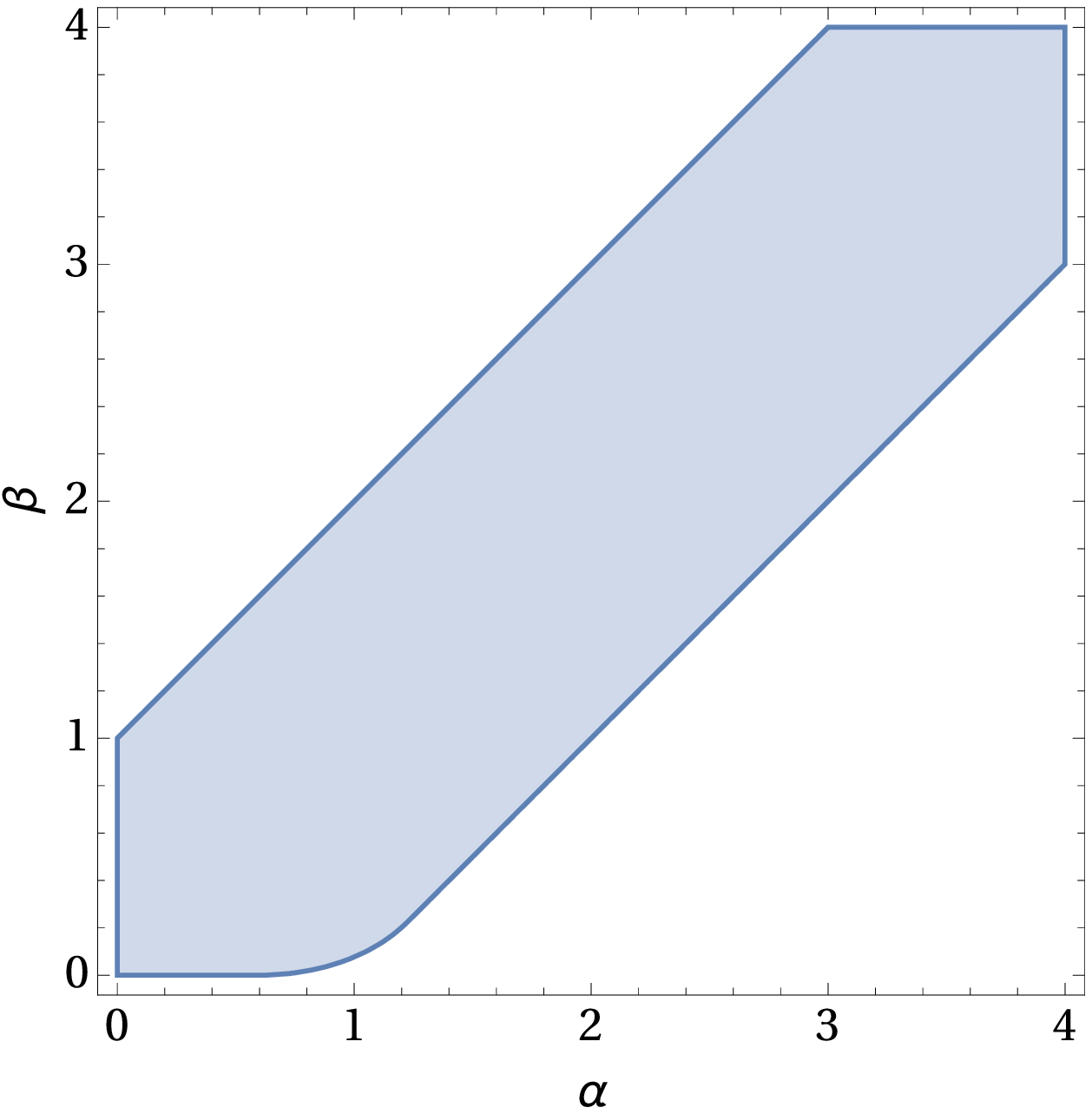}\hskip3mm
\includegraphics[width=65mm,height=55mm]{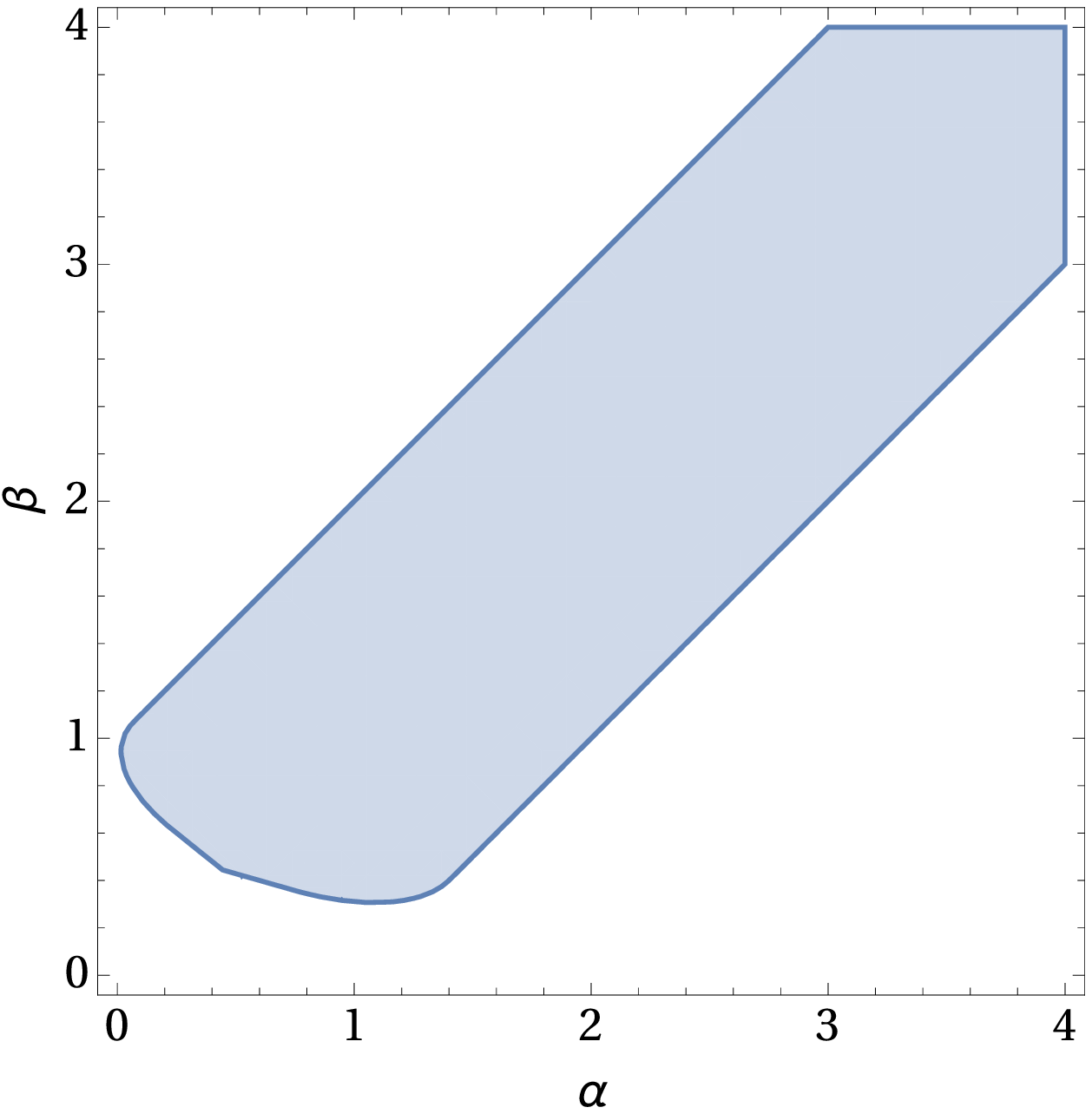} 
\includegraphics[width=65mm,height=55mm]{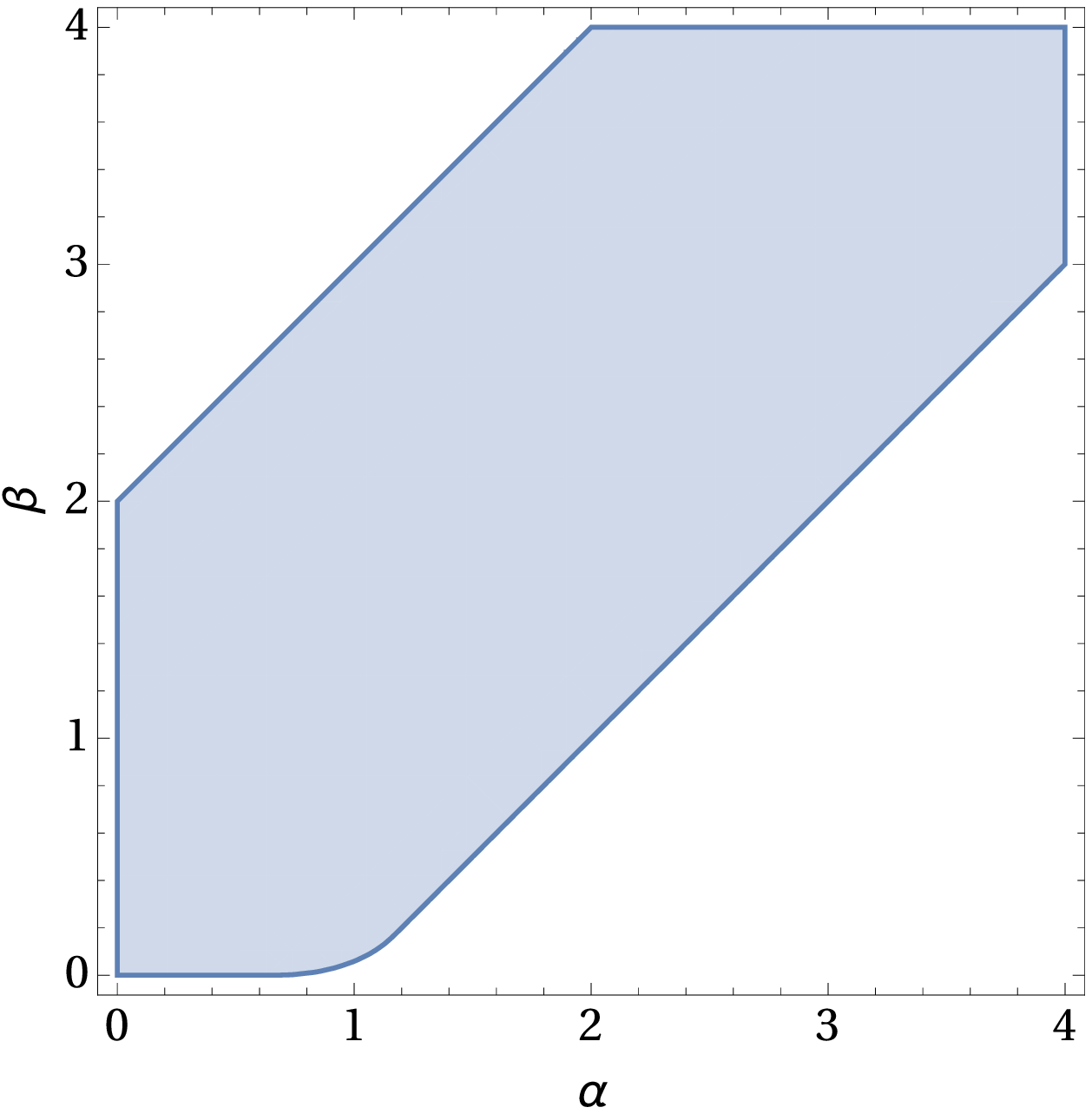} \hskip3mm
\includegraphics[width=65mm,height=55mm]{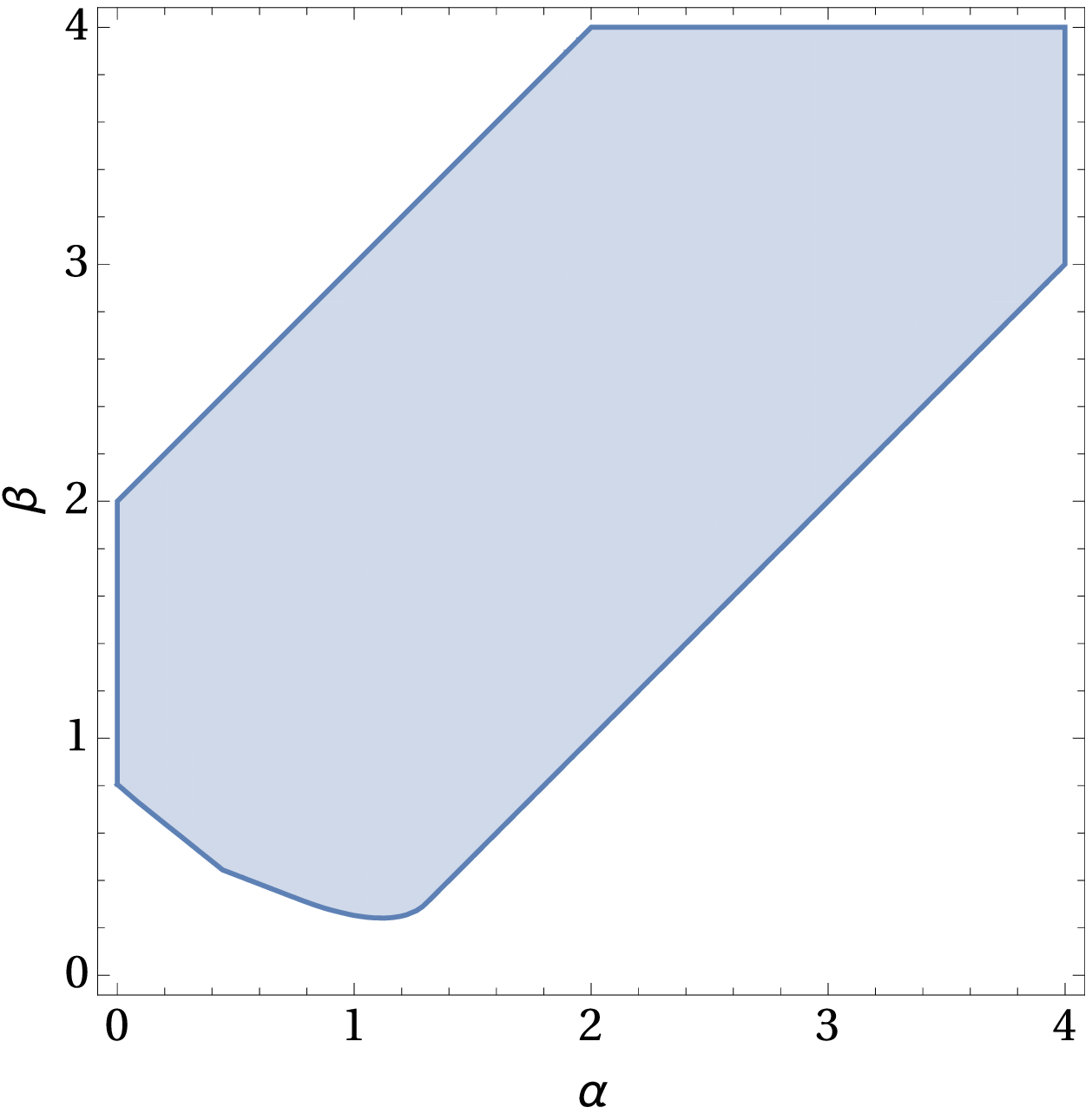} 
\includegraphics[width=65mm,height=55mm]{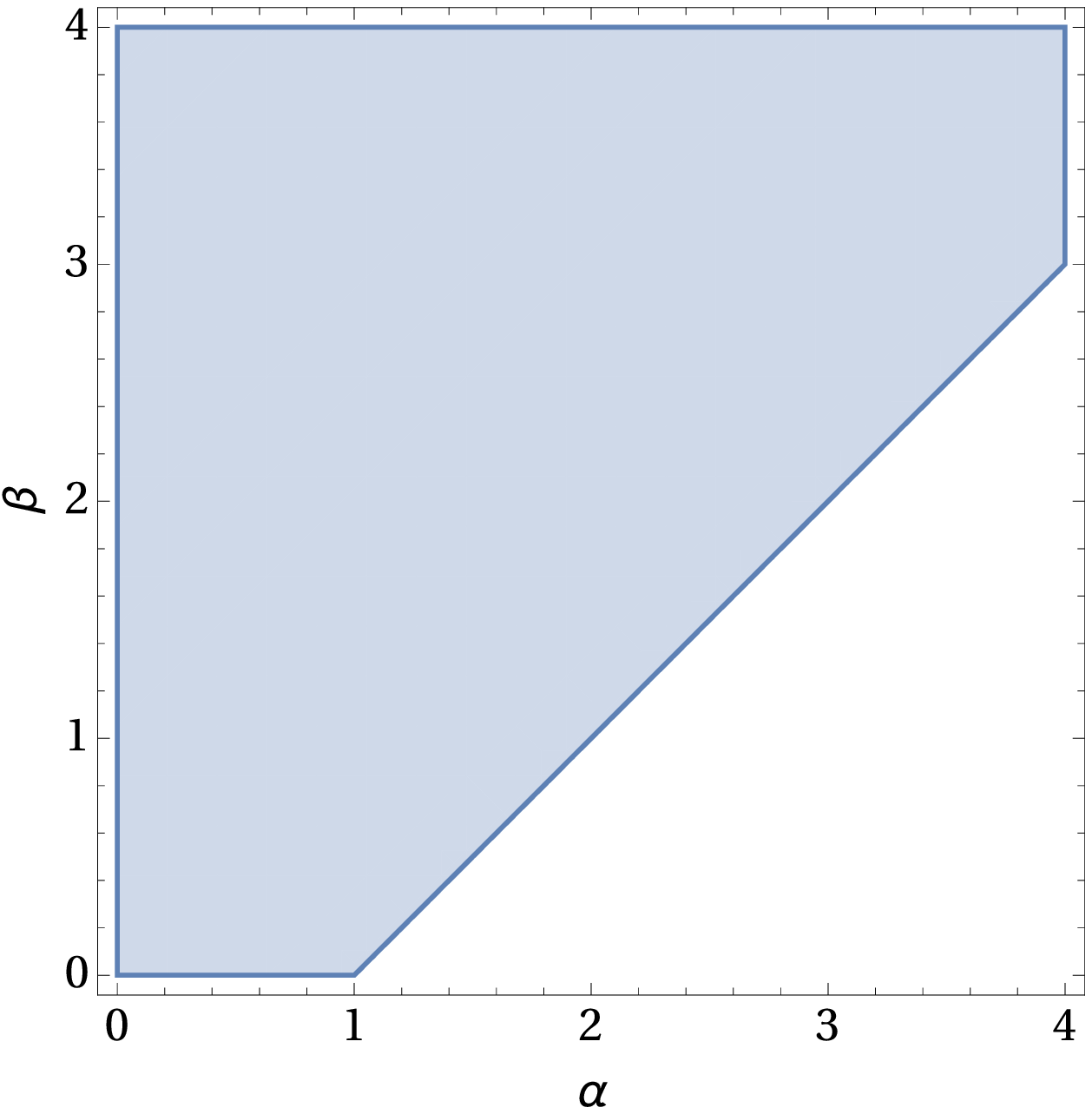} \hskip3mm
\includegraphics[width=65mm,height=55mm]{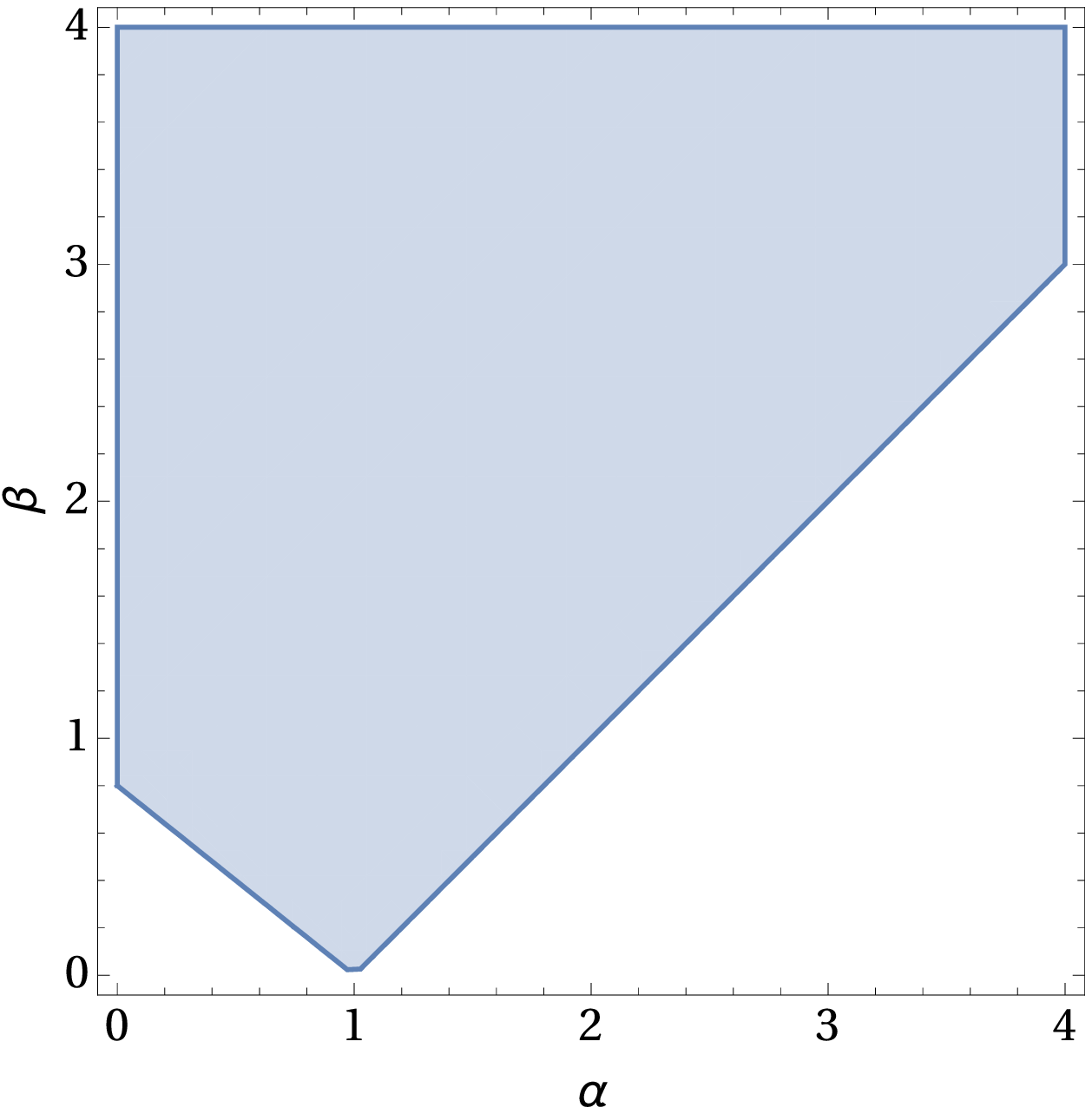} 
\caption{Set $R_0(d)$ of all $(\alpha,\beta)$ for which the zeroth-order entropy
is dissipating. Left column: $d=2$, right column: $d=10$.
Top row: explicit Euler scheme with $C_{\rm RK}=2$, 
middle row: implicit Euler scheme with $C_{\rm RK}=1$,
bottom row: Runge-Kutta scheme of order $p\ge 2$ with $C_{\rm RK}=0$.}
\label{fig.0th}
\end{figure}

\begin{proof}
Since $k\in\N$ is fixed, we set $u:=u^k$.
We choose the functions
$$
  \Gamma_1(u) = c_1\beta^2 u^{2\beta-\alpha-1}, \quad
	\Gamma_2(u) = c_2\beta^2 u^{2\beta-2\alpha-1}.
$$
It holds $h''(u)=u^{\alpha-1}$ and $\mu(u)=\beta u^{\beta-\alpha}$.
Then the coefficients in \eqref{de.ai} are as follows:
\begin{align*}
  a_1 &= \beta^2\big((C_{\rm RK}+1) + (1-\tfrac{1}{d})c_1\big)u^{2\beta-\alpha-1}, \\
  a_2 &= \beta^2\big((C_{\rm RK}+2)(\beta-\alpha)
	+ (1-\tfrac{1}{d})(2\beta-\alpha-1)c_1
	- (\tfrac{2}{d}+1)c_2\big)u^{2\beta-2\alpha-1}, \\
	a_3 &= \beta^2\big((\beta-\alpha)^2 - (2\beta-2\alpha-1)c_2\big)
	u^{2\beta-3\alpha-2}, \\
	a_4 &= -\beta^2(d-1)\big((2\beta-\alpha-1)c_1+2c_2\big)u^{2\beta-2\alpha-1}, \\
	a_5 &= -\beta^2 c_1 u^{2\beta-\alpha-1}, \quad
	a_6 = -\beta^2 d(d-1) c_1 u^{2\beta-\alpha-1}.
\end{align*}
Introducing the variables $\eta_j = \xi_j/u^{\alpha}$ for
$j\in\{G,L,R,S\}$, we can write \eqref{de.poly} as
\begin{align*}
  & G''(0)\le G''(0)+J_1+J_2 = -\beta^2\int_\Omega u^{2\beta+\alpha-1}Q(\eta)dx, \\
	&\mbox{where }
	Q(\eta) = b_1\eta_L^2 + b_2\eta_L\eta_G^2 + b_3\eta_G^4 
	+ b_4\eta_S\eta_G^2 + b_5\eta_R^2	+ b_6\eta_S^2
\end{align*}
with coefficients
\begin{align*}
  b_1 &= (C_{\rm RK}+1)+(1-\tfrac{1}{d})c_1, \\
	b_2 &= (C_{\rm RK}+2)(\beta-\alpha) + (1-\tfrac{1}{d})(2\beta-\alpha-1)c_1
	- (\tfrac{2}{d}+1)c_2, \\
	b_3 &= (\beta-\alpha)^2 - (2\beta-2\alpha-1)c_2, \\
	b_4 &= -(d-1)\big((2\beta-\alpha-1)c_1+2c_2\big), \\
	b_5 &= -c_1, \quad b_6 = -d(d-1) c_1.
\end{align*}

We need to determine all $(\alpha,\beta)$ 
such that there exist $c_1\le 0$, $c_2\in\R$ such that $Q(\eta)\ge 0$
for all $\eta=(\eta_G,\eta_L,\eta_R,\eta_S)$.
Without loss of generality, we exclude the cases $b_1=b_2=0$ and
$b_4=b_6=0$ since they lead to parameters $(\alpha,\beta)$ included in the
region calculated below. Thus, let $b_1>0$ and $b_6>0$. These inequalities
give the bound $-(C_{\rm RK}+1)/(1-1/d)<c_1<0$. Thus, we may introduce the parameter
$\lambda\in(0,1)$ by setting $c_1=-\lambda(C_{\rm RK}+1)/(1-1/d)$.
The polynomial $Q(\eta)$ can be rewritten as
\begin{align*}
  Q(\eta) &= b_1\left(\eta_L+\frac{b_2}{2b_1}\eta_G^2\right)^2
	+ b_6\left(\eta_S+\frac{b_4}{2b_6}\eta_G^2\right)^2 
	+ b_5\eta_R^2 + \eta_G^4\left(b_3-\frac{b_2^2}{4b_1}-\frac{b_4^2}{4b_6}\right) \\
  &\ge\eta_G^4\left(b_3-\frac{b_4^2}{4b_6}-\frac{b_2^2}{4b_1}\right)
	=: \frac{\eta_G^4 (C_{\rm RK}+1)}{4b_1b_6}R(c_2;\lambda,\alpha,\beta),
\end{align*}
where $R(c_2;\lambda,\alpha,\beta)$ is a quadratic polynomial in $c_2$ with
the nonpositive leading term $-d^2(4-3\lambda) + 4(2-3\lambda)d - 4$. The
polynomial $R(c_2;\lambda,\alpha,\beta)$ is nonnegative for some $c_2$ if
and only if its discriminant $4d^2\lambda(1-\lambda)S(\lambda;\alpha,\beta)$
is nonnegative. Here, $S(\lambda;\alpha,\beta)$ is a quadratic polynomial
in $\lambda$. In order to derive the conditions on $(\alpha,\beta)$ such that 
$S(\lambda;\alpha,\beta)\ge 0$ for some $\lambda\in(0,1)$, we employ
the computer-algebra system {\tt Mathematica}. The result of the command
\begin{quote}\begin{verbatim}
Resolve[Exists[LAMBDA, S[LAMBDA] >= 0 && LAMBDA > 0 
&& LAMBDA < 1], Reals]	
\end{verbatim}\end{quote}
gives all $(\alpha,\beta)\in\R^2$ such that there exist $c_1\le 0$, $c_2\in\R$ 
such that $Q(\eta)\ge 0$. The interior of this region equals the set $R_0(d)$, 
defined in the statement of the theorem. This shows that $G''(0)\le 0$ for
all $(\alpha,\beta)\in R_0(d)$.

If $G''(0)=0$, the nonnegative polynomial $Q$ has to vanish. 
In particular, $b_1\eta_L^2=0$. If $\eta_L=0$ in $\Omega$, the boundary
conditions imply that $u$ is constant, which contradicts our assumption
that $u$ is not the steady state. Thus $b_1=0$. Similarly, $b_2=b_3=b_4=0$.
This gives a system of four inhomogeneous linear equations for $(c_1,c_2)$ which is
unsolvable. Consequently, $G''(0)<0$.

The set $R_0(d)$ is nonempty since, e.g., $(1,1)\in R_0(d)$. Indeed, choosing
$c_1=-1$ and $c_2=0$, we find that $Q(\eta)=(C_{\rm RK}+\frac{1}{d})\eta_L^2
+ \eta_R^2 + d(d-1)\eta_S^2\ge 0$.

In one space dimension, the situation simplifies since the Laplacian coincides
with the Hessian and thus, the integration-by-parts formula \eqref{de.J1} 
is not needed. Then (see \eqref{de.J2})
$$
  G''(0) = G''(0)+J_1
	= -\beta^2\int_\Omega u^{2\beta+\alpha-1}\big(a_1\xi_L^2 + a_2\xi_L\xi_G^2
	+ a_3\xi_G^4\big)dx,
$$
where
$$
  a_1 = C_{\rm RK}+1, \quad 
	a_2 = (C_{\rm RK}+2)(\beta-\alpha) - 3c_2, \quad
	a_3 = (\beta-\alpha)^2 - (2\beta-2\alpha-1)c_2.
$$
The polynomial $P(\xi)=\xi_G^4(a_1 y^2 + a_2 y + a_3)$ with $y=\xi_L/\xi_G^2$
is nonnegative if and only if $a_1\ge 0$ and 
$4a_1a_3-a_2^2\ge 0$, which is equivalent to
\begin{equation}\label{pm.aux}
  -9c_2^2 + 2\big((C_{\rm RK}-2)(\alpha-\beta)+2(C_{\rm RK}+1)\big)c_2 
	- C_{\rm RK}^2(\alpha-\beta)^2 \ge 0.
\end{equation}
This inequality has a solution $c_2\in\R$ if and only if the quadratic 
polynomial has real roots, i.e.\ if its discriminant is nonnegative,
\begin{align*}
	0 &\le \big((C_{\rm RK}-2)(\alpha-\beta)+2(C_{\rm RK}+1)\big)^2
	- 9C_{\rm RK}^2(\alpha-\beta)^2 \\
	&= 4(C_{\rm RK}+1)\left(-(2C_{\rm RK}-1)(\alpha-\beta)^2
	+ (C_{\rm RK}-2)(\alpha-\beta) + (C_{\rm RK}+1)\right). 
\end{align*}
The polynomial $-(2C_{\rm RK}-1)z^2 + (C_{\rm RK}-2)z + (C_{\rm RK}+1)$ 
with $z=\alpha-\beta$ is always nonnegative if $C_{\rm RK}=0$ (implicit Euler).
For $C_{\rm RK}=1$ and $C_{\rm RK}=2$, this property holds if and only if
$-(C_{\rm RK}+1)/(2C_{\rm RK}-1)\le\alpha-\beta\le 1$.
This concludes the proof.
\end{proof}

%%%%%%%%%%%%%%%

\subsection{First-order entropies}

We consider the one-dimensional case and first-order entropies
with $f(u)=u^{\alpha/2}$, $\alpha>0$.

\begin{theorem}\label{thm.pm1}
Let $\Omega\subset\R$ be a bounded interval.
Let $(u^k)$ be a sequence of (smooth) solutions to the Runge-Kutta scheme
\eqref{1.rk} of order $p\ge 2$ for \eqref{pm.eq} in one space dimension.
Let the entropy be given by $F[u]=\int_\Omega (u^{\alpha/2})_x^2 dx$ with
$\alpha>0$, let $k\in\N$ be fixed,
and let $u^k$ be not the constant steady state of \eqref{pm.eq}.
There exists a nonempty region $R_1\in[0,\infty)^2$ and $\tau^k>0$ such that for all 
$(\alpha,\beta)\in R_1$, there is a constant $C_{\alpha,\beta}>0$ such that
for all $0<\tau\le\tau^k$,
$$
  F[u^k] + \tau C_{\alpha,\beta}\int_\Omega (u^k)^{\alpha+\beta-3}(u^k_{xx})^2 dx
	\le F[u^{k-1}], \quad k\in\N.
$$
\end{theorem}

Figure \ref{fig.1st} illustrates the set $R_1$. The set of admissible
values $(\alpha,\beta)$ for the continuous equation is given by
$\{-2\le\alpha-2\beta<1\}$ (the borders of this set are depicted in  the figure
by the dashed lines). 

\begin{figure}[ht]
\includegraphics[width=85mm]{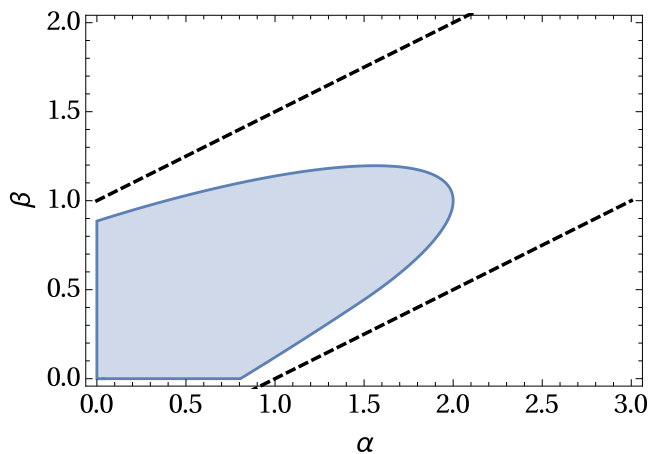}
\caption{Set of all $(\alpha,\beta)$ for which the discrete first-order entropy
for solutions to the one-dimensional porous-medium equation is dissipating.
The continuous first-order entropy is dissipated for $-2\le\alpha-2\beta<1$.
The borders of this set is indicated in the figure by dashed lines.}
\label{fig.1st}
\end{figure}

\begin{proof}
First, we compute $G'(0)$ according to Theorem \ref{thm.ed2}:
$$
  G'(0) = -\alpha\int_\Omega u^{\alpha/2-1}(u^{\alpha/2})_{xx}
	(u^\beta)_{xx} dx.
$$
We show that $G'(0)$ is nonpositive in a certain range of values
$(\alpha,\beta)$. We formulate $G'(0)$ as
$$
  G'(0) = -\frac{\alpha^2\beta}{4}\int_\Omega u^{\alpha+\beta-1}
	\big((\alpha-2)(\beta-1)\xi_1^4 + (\alpha+2\beta-4)\xi_1^2\xi_2 + 2\xi_2^2\big)dx,
$$
where $\xi_1=u_x/u$, $\xi_2=u_{xx}/u$.
We employ the integration-by-parts formula
$$
  0 = \int_\Omega(u^{\alpha+\beta-4}u_x^3)_x dx
	= \int_\Omega u^{\alpha+\beta-1}\big((\alpha+\beta-4)\xi_1^4 + 3\xi_1^2\xi_2\big)dx
	=: J.
$$
Therefore,
$$
  G'(0) = G'(0)-\frac{\alpha^2\beta}{4}cJ
	= -\frac{\alpha^2\beta}{4}\int_\Omega u^{\alpha+\beta-1}P(\xi)dx,
$$
where
$$
  P(\xi) = \big((\alpha-2)(\beta-1)+(\alpha+\beta-4)c\big)\xi_1^4  
	+ \big(\alpha+2\beta-4+3c\big)\xi_1^2\xi_2 + 2\xi_2^2.
$$
This polynomial is nonnegative if and only if 
$$
  8\big((\alpha-2)(\beta-1)+(\alpha+\beta-4)c\big) 
	- (\alpha+2\beta-4+3c)^2\ge 0,
$$
which is equivalent to
$$
  g(c) := -9c^2 + 2(\alpha-2\beta-4)c - (\alpha-2\beta)^2 \ge 0.
$$
The maximizing value $c^*=(\alpha-2\beta-4)/9$, obtained
from $g'(c)=0$, yields
$$
  g(c^*) = -\frac89(\alpha-2\beta-1)(\alpha-2\beta+2) \ge 0
$$
and consequently $G'(0)\le 0$ if $-2\le\alpha-2\beta\le 1$. 
This condition is the same as in \cite[Theorem 13]{CJS15} for the
continuous equation.

Next, we turn to the proof of $G''(0)<0$. The proof of Theorem \ref{thm.ed2}
shows that
\begin{align*}
	G''(0) &= -\frac{\alpha}{2}\int_\Omega\bigg(\frac{\alpha}{2}
	\big(u^{\alpha/2-1}(u^\beta)_{xx}\big)_x^2 - \bigg(\frac{\alpha}{2}-1\bigg)
	u^{\alpha/2-2}(u^{\alpha/2})_{xx}(u^\beta)_{xx}^2 \\
	&\phantom{xx}{}- \beta C_{\rm RK}u^{\alpha/2-1}(u^{\alpha/2})_{xx}
	\big(u^{\beta-1}(u^\beta)_{xx}\big)_{xx}\bigg)dx.
\end{align*}
We integrate by parts in the last term and use
$(\beta u^{\beta-1}(u^\beta)_{xx})_x=0$ on $\pa\Omega$:
\begin{align*}
  G''(0) &= -\frac18\alpha^2\beta^2\int_\Omega u^{\alpha+2\beta-2} \\
	&\phantom{xx}{}\times
	\big(a_1\xi_1^6 + a_2\xi_1^4\xi_2 + a_3\xi_1^3\xi_3 + a_4\xi_1^2\xi_2^2
	+ a_5\xi_1\xi_2\xi_3 + a_6\xi_2^3 + a_7\xi_3^2\big)dx,
\end{align*}
where $\xi_1=u_x/u$, $\xi_2=u_{xx}/u$, $\xi_3=u_{xxx}/u$, and
\begin{align*}
  a_1 &= (\beta-1)\big(2C_{\rm RK}\alpha^2\beta - 3C_{\rm RK}\alpha^2
	+ 2\alpha\beta^2-2(5C_{\rm RK}+3)\alpha\beta + (15C_{\rm RK}+4)\alpha \\
	&\phantom{xx}{}+ 2\beta^3 - 14\beta^2 + 4(3C_{\rm RK}+7)\beta 
	- 2(9C_{\rm RK} + 8)\big), \\
	a_2 &= (\beta-1)\big(4C_{\rm RK}\alpha^2 + (8C_{\rm RK}+7)\alpha\beta
	- (32C_{\rm RK}+9)\alpha + 12\beta^2 - 2(8C_{\rm RK}+25)\beta \\
	&\phantom{xx}{}+ 6(8C_{\rm RK} + 7)\big), \\
	a_3 &= C_{\rm RK}\alpha^2 + 2\alpha\beta - (5C_{\rm RK}+2)\alpha
	+ 4(C_{\rm RK}+1)\beta^2 - 2(5C_{\rm RK}+8)\beta + 12(C_{\rm RK} + 1), \\
	a_4 &= 2(\beta-1)\big(2(4C_{\rm RK}+1)\alpha + 9\beta - (16C_{\rm RK} + 13)\big), \\
	a_5 &= 2(2C_{\rm RK}+1)\alpha + 4(2C_{\rm RK}+3)\beta - 16(C_{\rm RK}+1), \\
	a_6 &= 2-\alpha, \quad a_7 = 2(C_{\rm RK}+1).
\end{align*}
We employ three integration-by-parts formulas:
\begin{align*}
  0 &= \int_\Omega\big(u^{\alpha+2\beta-5}u_{xx}^2 u_x\big)_x dx 
	= \int_\Omega u^{\alpha+2\beta-2}\big((\alpha+2\beta-5)\xi_1^2\xi_2^2
	+ 2\xi_1\xi_2\xi_3 + \xi_2^3\big)dx =: J_1, \\
	0 &= \int_\Omega\big(u^{\alpha+2\beta-6}u_{xx}u_x^3\big)_x dx
	= \int_\Omega u^{\alpha+2\beta-2}\big((\alpha+2\beta-6)\xi_1^4\xi_2
	+ \xi_1^3\xi_3 + 3\xi_1^2\xi_2^2\big)dx =: J_2, \\
	0 &= \int_\Omega\big(u^{\alpha+2\beta-7}u_x^5\big)_x dx
	= \int_\Omega u^{\alpha+2\beta-2}\big((\alpha+2\beta-7)\xi_1^6
	+ 5\xi_1^4\xi_2\big)dx =: J_3.
\end{align*}
Then
\begin{align*}
  & G''(0) = G''(0)-\frac18\alpha^2\beta^2(c_1J_1+c_2J_2+c_3J_3) 
	= -\frac18\alpha^2\beta^2\int_\Omega u^{\alpha+2\beta-2}P(\xi)dx, \\
  &\mbox{where }P(\xi) = b_1\xi_1^6 + b_2\xi_1^4\xi_2	+ b_3\xi_1^3\xi_3 
	+ b_4\xi_1^2\xi_2^2 + b_5\xi_1\xi_2\xi_3 + b_6\xi_2^3 + b_7\xi_3^2,
\end{align*}
and the coefficients are given by
\begin{align*}
  b_1 &= a_1 + (\alpha+2\beta-7)c_3,  &
	b_2 &= a_2 + (\alpha+2\beta-6)c_2 + 5c_3, \\
	b_3 &= a_3 + c_2, &	b_4 &= a_4 + (\alpha+2\beta-5)c_1 + 3c_2, \\
	b_5 &= a_5 + 2c_1,  & b_6 &= a_6 + c_1, \\
	b_7 &= a_7. 
\end{align*}
Choosing $c_1=-a_6$, we eliminate the cubic term $\xi_2^3$. 
Furthermore, setting, $x=\xi_2/\xi_1^2$ and $y=\xi_3/\xi_1^3$, we can write
the polynomial $P$ as a quadratic polynomial in $(x,y)$:
$$
  Q(x,y)=\xi_1^6 P(\xi)
	= b_1 + b_2 x + b_3 y + b_4 x^2 + b_5 xy + b_7 y^2.
$$

The following lemma is a consequence of the proof of Lemma 2.2 in \cite{JuMi09}.

\begin{lemma}
The polynomial $p(x,y) = A + B x + C y + D x^2 + E xy + F y^2$ with $F>0$
is nonnegative for all $(x,y)\in\R^2$ if and only if 
\begin{align*}
  &\text{\rm (i)}\phantom{.}\quad 4DF-E^2>0 \quad\mbox{and}\quad
	A(4DF-E^2) - B^2F - C^2D + BCE \ge 0, \quad\mbox{or} \\
  &\text{\rm (ii)}\quad 4DF-E^2=0\quad\mbox{and}\quad 2BF-CE=0\quad\mbox{and}\quad 
	4AF-C^2\ge 0.
\end{align*}
\end{lemma}

Note that in case $4DF-E^2=0$ and $E\neq 0$, we may replace $2BF-CE=0$ 
by the condition $2BEF = CE^2 = 4CDF$ or (since $F>0$) $BE=2CD$.

The first inequality in case (i),
\begin{align*}
  0 &< 4b_4b_7 - b_5^2
	= -(C_{\rm RK}+1)(2C_{\rm RK}+1)\alpha^2 
	+ (2C_{\rm RK}+2)(4C_{\rm RK}-3)\alpha\beta
	+ (9C_{\rm RK}+9)\alpha \\
	&\phantom{xx}{}- 2C_{\rm RK}(4C_{\rm RK}+3)\beta^2
	+ (8C_{\rm RK}+12)\beta + (3C_{\rm RK}+3)c_2 - (12C_{\rm RK} + 14),
\end{align*}
is linear in $c_2$ and provides a lower bound for $c_2$:
\begin{align*}
  c_2 &> \frac{1}{3(C_{\rm RK}+1)}\Big((C_{\rm RK}+1)(2C_{\rm RK}+1)\alpha^2 
	- (2C_{\rm RK}+2)(4C_{\rm RK}-3)\alpha\beta
	- (9C_{\rm RK}+9)\alpha \\
	&\phantom{xx}{}+ 2C_{\rm RK}(4C_{\rm RK}+3)\beta^2
	- (8C_{\rm RK}+12)\beta + (12C_{\rm RK} + 14)\Big) =: c_2^*.
\end{align*}
The second inequality in case (i) becomes
$$
  0\le b_1(4b_4b_7-b_5^2) - b_2^2b_7 - b_3^2b_4 + b_2b_3b_5
  = -50(C_{\rm RK}+1)c_3^2 + p_1(\alpha,\beta,c_2)c_3 + p_2(\alpha,\beta,c_2),
$$
where $p_1$ and $p_2$ are some polynomials in $\alpha$, $\beta$, and $c_2$.
This quadratic expression in $c_3$ is nonnegative if and only if its discriminant
is nonnegative,
\begin{align*}
  0 &\le -200(C_{\rm RK}+1)p_2(\alpha,\beta,c_2) - p_1(\alpha,\beta,c_2)^2 \\
	&= -8\big(4b_4b_7 - b_5^2\big)
	\big(25c_2^2 + p_3(\alpha,\beta)c_2	+ p_4(\alpha,\beta)\big),
\end{align*}
where $p_3(\alpha,\beta)$ and $p_4(\alpha,\beta)$ are
some polynomials in $\alpha$ and $\beta$. The factor $4b_4b_7 - b_5^2$ is positive,
so we have to ensure that 
$R_{\alpha,\beta}(c_2)=25c_2^2 + p_3(\alpha,\beta)c_2	+ p_4(\alpha,\beta)\le 0$
for some $c_2>c_2^*$. 
Therefore we must ensure that the rightmost root of $R_{\alpha,\beta}(c_2)$ is 
larger or equal than the lower bound for $c_2$, i.e., 
$-p_3(\alpha,\beta)+\sqrt{p_3^2(\alpha,\beta)-100p_4(\alpha,\beta)}\ge 50c^*_2$.
For $C_{\rm RK}=1$, 
the values $(\alpha,\beta)$ for which there exists $c_2>c_2^*$ such that
$R_{\alpha,\beta}(c_2)\le 0$ is depicted in Figure \ref{fig.1st}.
In case (ii), we may immediately calculate $c_2$ and $c_3$ but this results in
a region which is already contained in the first one.
This shows that $G''(0)\le 0$.

If $G''(0)=0$, the polynomial $Q$ vanishes. Thus, either $u_x/u=\xi_1=0$ or
$P(\xi)=0$ in $\Omega$. The first case is impossible since $u$ is not constant
in $\Omega$. As $b_7=a_7=2(C_{\rm RK}+1)>0$, the second case $P(\xi)=0$ implies
that $\xi_3=0$. Hence, $u$ is a quadratic polynomial. In view of the boundary
conditions, $u$ must be constant, but this contradicts our assumption. 
Hence, $G''(0)<0$.
\end{proof}

%%%%%%%%%%%%%%%%%%%%%%%%%%%%%%%%%%%%%%%%%%%%%%%%%%%%%%%%%%%%%%%%%%%%%%%%%%%%%%

\section{Linear diffusion system}\label{sec.sys}

We consider the following linear diffusion system:
\begin{equation}\label{sys.eq}
  \pa_t u_1 - \rho_1\Delta u_1 = \mu(u_2-u_1), \quad
  \pa_t u_2 - \rho_2\Delta u_2 = \mu(u_1-u_2),
\end{equation}
with initial and homogeneous Neumann boundary conditions,
$\rho_1$, $\rho_2$, $\mu>0$, and the entropy
\begin{equation}\label{sys.H}
  H[u] = \int_\Omega h(u)dx 
	= \int_\Omega\sum_{i=1}^2 u_i(\log u_i-1)dx,
\end{equation}
where $u=(u_1,u_2)$.
% and $\beta_1=1/(\rho_1+3\rho_2)$, $\beta_2=1/(\rho_2+3\rho_1)$.
If the initial data is nonnegative, the maximum principle
shows that the solutions to \eqref{sys.eq} are nonnegative too.

\begin{theorem}\label{thm.sys}
Let $(u^k)$ be a sequence of (smooth) nonnegative solutions to the Runge-Kutta scheme
\eqref{1.rk} for \eqref{sys.eq} with $C_{\rm RK}=1$ and $\rho:=\rho_1=\rho_2$.
Let the entropy $H$ be given by \eqref{sys.H}. Let $k\in\N$ be fixed and let
$u^k$ be not the steady state of \eqref{1.rk}. 
Then there exists $\tau^k>0$ such that for all $0<\tau<\tau^k$,
$$
  H[u^k] + \tau\int_\Omega\bigg(\rho\sum_{i=1}^2\frac{|\na u_i^k|^2}{u_i^k}
	+ \mu(\log u_1^k-\log u_2^k)(u_1^k-u_2^k)\bigg)dx \le H[u^{k-1}].
$$
\end{theorem}

Note that we need equal diffusivities $\rho_1=\rho_2$ and higher-order
schemes ($C_{\rm RK}=1$). These conditions are in accordance of \cite{LiMi13},
where the continuous equation was studied.
In order to highlight the step where these conditions are needed, the
following proof is slightly more general than actually needed.

\begin{proof}
We fix $k\in\N$ and set $u:=u^k$. Let $A[u]=(A_1[u],A_2[u])
=(\rho_1\Delta u_1+\mu(u_2-u_1),\rho_2\Delta u_2+\mu(u_1-u_2))$.
Since $A$ is linear, $DA[u](h)=A[h]$. Thus,
$$
  G''(0) = -\int_\Omega\big(C_{\rm RK}
	h'(u)^\top A[A[u]] + A[u]^\top h''(u)A[u]\big)dx 
  = -G_{1} - G_{2}.
$$
In the following, we set $\pa_i h=\pa h/\pa u_i$ for $i=1,2$.
We integrate by parts twice, using the boundary conditions
$\na u_i\cdot\nu=0$ and $\na A_i[u]\cdot\nu=0$ on $\pa\Omega$, and collect
the terms:
\begin{align*}
  G_{1} &= C_{\rm RK}\int_\Omega\Big(
	\pa_1 h(u)\big(\rho_1\Delta A_1[u] + \mu(A_2[u]-A_1[u])\big) \\
	&\phantom{xx}{}
	+ \pa_2 h(u)\big(\rho_2\Delta A_2[u] + \mu(A_1[u]-A_2[u])\big)\Big)dx \\
	&= C_{\rm RK}\int_\Omega\Big(\rho_1\Delta\pa_1 h(u)A_1[u]
	+ \rho_2\Delta\pa_2 h(u)A_2[u] \\
	&\phantom{xx}{}+ \mu(\pa_1 h(u)-\pa_2 h(u))(A_2[u]-A_1[u])\Big)dx \\
	&= C_{\rm RK}\int_\Omega\Big(\rho_1\big(\pa_1^2 h(u)\Delta u_1
	+\pa_1^3 h(u)|\na u_1|^2\big)\big(\rho_1 \Delta u_1+\mu(u_2-u_1)\big) \\
	&\phantom{xx}{}+ \rho_2\big(\pa_2^2 h(u)\Delta u_2
	+ \pa_2^3 h(u)|\na u_2|^2\big)\big(\rho_2 \Delta u_2+\mu(u_1-u_2)\big) \\
	&\phantom{xx}{}+ \mu(\pa_2 h(u)-\pa_1 h(u))\big(\rho_1\Delta u_1-\rho_2\Delta u_2
	+ 2\mu(u_2-u_1)\big)\Big)dx \\
	&= C_{\rm RK}\int_\Omega\Big(\rho_1^2\pa_1^2h(u)(\Delta u_1)^2
	+ \rho_2^2\pa_2^2h(u)(\Delta u_2)^2 
	+ \rho_1^2\pa_1^3h(u)\Delta u_1|\na u_1|^2 \\
	&\phantom{xx}{}+ \rho_2^2\pa_2^3h(u)\Delta u_2|\na u_2|^2 
	+ \rho_1 \mu\big(\pa_1^2h(u)(u_2-u_1) + \pa_2h(u)-\pa_1h(u)\big)\Delta u_1 \\
	&\phantom{xx}{}
	+ \rho_2 \mu\big(\pa_2^2h(u)(u_1-u_2) + \pa_1h(u)-\pa_2h(u)\big)\Delta u_2 
	+ \rho_1 \mu \pa_1^3h(u)(u_2-u_1)|\na u_1|^2 \\
	&\phantom{xx}{}+ \rho_2 \mu \pa_2^3h(u)(u_1-u_2)|\na u_2|^2
	+ 2\mu^2(\pa_2h(u)-\pa_1h(u))(u_2-u_1)\Big)dx.
\end{align*}
Furthermore,
\begin{align*}
  G_{2} &= \int_\Omega\Big(\pa_1^2h(u)\big(\rho_1\Delta u_1+\mu(u_2-u_1)\big)^2
	+ \pa_2^2h(u)\big(\rho_2\Delta u_2+\mu(u_1-u_2)\big)^2\Big)dx \\
	&= \int_\Omega\Big(\rho_1^2\pa_1^2h(u)(\Delta u_1)^2
	+ \rho_2^2\pa_2^2h(u)(\Delta u_2)^2 + 2\rho_1 \mu\pa_1^2h(u)(u_2-u_1)\Delta u_1 \\
	&\phantom{xx}{}+ 2\rho_2 \mu\pa_2^2h(u)(u_1-u_2)\Delta u_2
	+ \mu^2(\pa_1^2h(u)+\pa_2^2h(u))(u_1-u_2)^2\Big)dx.
\end{align*}
Adding $G_1$ and $G_2$, we arrive at
\begin{align*}
  G''(0) &= -\sum_{i=1}^2\int_\Omega\Big(\rho_i^2(C_{\rm RK}+1)\pa_i^2h(u)
	(\Delta u_i)^2 + \rho_i^2C_{\rm RK}\pa_i^3h(u)\Delta u_i|\na u_i|^2\Big)dx \\
	&\phantom{xx}{}- \int_\Omega\Big(\rho_1\mu\big((C_{\rm RK}+2)\pa_1^2h(u)(u_2-u_1)
	+ C_{\rm RK}(\pa_2h(u)-\pa_1h(u))\big)\Delta u_1 \\
	&\phantom{xx}{}+ \rho_2\mu\big((C_{\rm RK}+2)\pa_2^2h(u)(u_1-u_2)
	+ C_{\rm RK}(\pa_1h(u)-\pa_2h(u))\big)\Delta u_2 \\
	&\phantom{xx}{} + \rho_1 \mu C_{\rm RK}\pa_1^3h(u)(u_2-u_1)|\na u_1|^2
	+ \rho_2 \mu C_{\rm RK}\pa_2^3h(u)(u_1-u_2)|\na u_2|^2\Big)dx \\
	&\phantom{xx}{}
	- \int_\Omega \mu^2\Big(2(\pa_1h(u)-\pa_2h(u))
	+(\pa_1^2h(u)+\pa_2^2h(u))(u_1-u_2)\Big)(u_1-u_2)dx \\
	&= -I_2 - I_1 - I_0.
\end{align*}

The idea of \cite{LiMi13}
is to show that each integral $I_i$, involving only derivatives of order $i$, 
is nonnegative. In contrast to \cite{LiMi13}, we employ systematic
integration by parts, which allows for a simpler and more general proof in our context.
For the term $I_2$, we use the following integration-by-parts formula:
$$
  0 = \int_\Omega\diver\big(u_i^{-2}|\na u_i|^3\big)dx
	= \int_\Omega\big(-2u_i^{-3}|\na u_i|^4 + 3u_i^{-2}\Delta u_i|\na u_i|^2\big)dx
	=: J_i.
$$
Then, for $\eps>0$,
\begin{align*}
  &I_2 - c\sum_{i=1}^2\rho_i^2 J_i - \eps\sum_{i=1}^2 u_i^{-3}|\na u_i|^4dx \\
	&= \sum_{i=1}^2\rho^2_i\int_\Omega\Big((C_{\rm RK}+1)u_i^{-1}(\Delta u_i)^2
	- (3c + C_{\rm RK})u_i^{-2}\Delta u_i|\na u_i|^2 
	+ (2c-\eps) u_i^{-3}|\na u_i|^4\Big)dx.
\end{align*}
The integrand defines a quadratic polynomial in the variables $\Delta u_i$ 
and $|\na u_i|^2$ and is nonnegative if its discriminant satisfies
$4(2c-\eps)(C_{\rm RK}+1)-(3c+C_{\rm RK})^2\ge 0$. It turns out that this inequality
holds true for $C_{\rm RK}\in\{0,1\}$ if we choose $c=2/3$ and $\eps>0$ sufficiently
small. When $C_{\rm RK}=2$, we can show only that $I_2\ge 0$ which is not
sufficient to prove that $G''(0)<0$ (see below). We conclude that
\begin{equation}\label{sys.I2}
  I_2\ge \eps\sum_{i=1}^2\int_\Omega u_i^{-3}|\na u_i|^4 dx.
\end{equation}

Integrating by parts in $I_1$ in order to obtain only first-order derivatives, 
we find after some rearrangements that
\begin{align*}
  I_1 &= \mu\int_\Omega\big(a_1|\na\log u_1|^2 + a_2\na\log u_1\cdot\na\log u_2
	+ a_3|\na\log u_2|^2\big)dx, \quad\mbox{where} \\
	a_1 &= 2\rho_1(C_{\rm RK} u_1+u_2), \quad
	a_3 = 2\rho_2(C_{\rm RK}u_2+u_1), \\
	a_2 &= -(C_{\rm RK}(\rho_1+\rho_2)+2\rho_2)u_1
	- (C_{\rm RK}(\rho_1+\rho_2)+2\rho_1)u_2.
\end{align*}
The integrand is nonnegative if and only if $4a_1a_3 - a_2^2\ge 0$
for all $(u_1,u_2)$. We compute:
\begin{align*}
  C_{\rm RK}=0: &\quad 4a_1a_3-a_2^2 = -4(\rho_1u_2-\rho_2u_1)^2, \\
	C_{\rm RK}=1: &\quad 4a_1a_3-a_2^2 = (\rho_1-\rho_2)\big(\rho_1(u_1^2+6u_1u_2+9u_2^2)
	- \rho_2(9u_1^2+6u_1u_2+u_2^2)\big), \\
	C_{\rm RK}=2: &\quad 4a_1a_3-a_2^2 = -4\big(\rho_1(u_1+2u_2) - \rho_2(2u_1+u_2)\big).
\end{align*}
Thus, $4a_1a_3 - a_2^2\ge 0$ is possible only if $\rho_1=\rho_2$ and
$C_{\rm RK}=1$. 

Finally, we see immediately that the remaining term
$$
  I_0 = \mu^2\int_\Omega\bigg(2(\log u_1-\log u_2)(u_1-u_2)
	+ \bigg(\frac{1}{u_1}+\frac{1}{u_2}\bigg)(u_1-u_2)^2\bigg)dx
$$
is nonnegative. This shows that $G''(0)\le 0$.
If $G''(0)=0$, we infer from \eqref{sys.I2} that $u_i=\mbox{const.}$, but this
contradicts our hypothesis that $u_i$ is not a steady state.
\end{proof}

%%%%%%%%%%%%%%%%%%%%%%%%%%%%%%%%%%%%%%%%%%%%%%%%%%%%%%%%%%%%%%%%%%%%%%%%%%%%%%

\section{The Derrida-Lebowith-Speer-Spohn equation}\label{sec.dlss}

Consider the one-dimensional fourth-order equation
\begin{equation}\label{dlss.eq}
  \pa_t u = -(u(\log u)_{xx})_{xx} \quad\mbox{in }\Omega, \ t>o, \quad
	u(0)=u^0
\end{equation}
with periodic boundary conditions. This equation appears as a scaling limit of
the so-called (time-discrete) Toom model, which describes interface fluctuations
in a two-dimensional spin system \cite{DLSS91}. The variable $u$ is the limit of
a random variable related to the deviation of the spin interface from a straight line.
The multi-dimensional version of \eqref{dlss.eq} models the eectron density $u$
in a quantum semiconductor, und the equation is the zero-temperature,
zero-field approximation of the quantum drift-diffusion model \cite{Jue09}.
For existence results for \eqref{dlss.eq}, we refer to \cite{JuMa08} and
references therein.

To simplify our calculations, we analyze
only the logarithmic entropy $H[u]=\int_\Omega u(\log u-1)dx$. It is
possible to verify condition \eqref{1.I0} also for entropies of the form
$\int_\Omega u^\alpha dx$, but it turns out that only sufficiently small
$\alpha>0$ are admissible (about $0<\alpha<0.15\ldots$) 
and the computations are very tedious.
Therefore, we restrict ourselves to the case $\alpha=0$.

\begin{theorem}
Let $(u^k)$ be a sequence of (smooth) solutions to the Runge-Kutta scheme
\eqref{1.rk} with $C_{\rm RK}=1$ for \eqref{dlss.eq}.
Let the entropy be given by $H[u]=\int_\Omega u(\log u-1)dx$, 
let $k\in\N$ be fixed, and let $u^k$ be not a steady state. 
Then there exists $\tau^k>0$ such that for all $0<\tau<\tau^k$,
$$
  H[u^k] + \tau q\int_\Omega u(\log u)_x^8 dx 
	+ \tau\int_\Omega u(\log u)_{xx}^2 dx \le H[u^{k-1}], \quad 
	q\approx 0.0045.
$$
\end{theorem}

\begin{proof}
First, we observe that $G'(0)=-\int_\Omega (u(\log u)_{xx})_{xx}\log udx
=-\int_\Omega u(\log u)_{xx}^2 dx$. With $A[u]=(u(\log u)_{xx})_{xx}$ and
$DA[u](h) = \big(h_{xx} - 2(\log u)_xh_x + (\log u)_x^2 h\big)_{xx}$,
we can write $G''(0)=-I_0^k$ according to \eqref{1.I0} as
\begin{align*}
  G''(0) &= -\int_\Omega\Big(\log u\big(A[u]_{xx}
	- 2(\log u)_xA[u]_x + (\log u)_x^2A[u]\big)_{xx} + \frac{1}{u}A[u]^2\Big)dx \\
	&= -\int_\Omega\Big((\log u)_{xx}\big(A[u]_{xx}
	- 2(\log u)_xA[u]_x + (\log u)_x^2A[u]\big) + \frac{1}{u}A[u]^2\Big)dx \\
	&= -\int_\Omega\Big(\big(v_{xxxx} + 2(v_x v_{xx})_x 
	+ v_x^2 v_{xx}\big)
	A[u] + \frac{1}{u}A[u]^2\Big)dx,
\end{align*}
where we have integrated by parts several times and have set $v=\log u$.
Then $A[u]=u(v_x^2 v_{xx} + 2v_x v_{xxx} + v_{xx}^2
+ v_{xxxx})$ and, with the abbreviations $\xi_1=v_x,\ldots,\xi_4=v_{xxxx}$,
\begin{align*}
  G''(0) &= -\int_\Omega u\Big(2\xi_1^4\xi_2^2 
	+ 8\xi_1^3\xi_2\xi_3
	+ 5\xi_1^2\xi_2^3 + 4\xi_1^2\xi_2\xi_4
	+ 8\xi_1^2\xi_3^2 + 10\xi_1\xi_2^2\xi_3 \\
	&\phantom{xx}{}+ 8\xi_1\xi_3\xi_4 + 3\xi_2^4
	+ 5\xi_2^2\xi_4 + 2\xi_4^2\Big)dx.
\end{align*}
We employ the following integration-by-parts formulas:
\begin{align*}
  0 &= \int_\Omega(uv_x^7)_x dx = \int_\Omega u(\xi_1^8 + 7\xi_1^6\xi_2)dx =: J_1, \\
	0 &= \int_\Omega(uv_{xx}v_x^5)_x dx
	= \int_\Omega u(\xi_1^6\xi_2 + \xi_1^5\xi_3 + 5\xi_1^4\xi_2^2)dx =: J_2, \\
	0 &= \int_\Omega(uv_{xxx}v_x^4)_x dx
	= \int_\Omega u(\xi_1^5\xi_3 + \xi_1^4\xi_4 + 4\xi_1^3\xi_2\xi_3)dx =: J_3, \\
	0 &= \int_\Omega(uv_{xx}^2v_x^3)_x dx
	= \int_\Omega u(\xi_1^4\xi_2^2 + 2\xi_1^3\xi_2\xi_3 + 3\xi_1^2\xi_2^3)dx =: J_4, \\
	0 &= \int_\Omega(uv_{xx}v_{xxx}v_x^2)_x dx
	= \int_\Omega u(\xi_1^3\xi_2\xi_3 + \xi_1^2\xi_2\xi_4 + \xi_1^2\xi_3^2
	+ 2\xi_1\xi_2^2\xi_3)dx =: J_5, \\
	0 &= \int_\Omega(uv_{xxx}^2v_x)_x dx
	= \int_\Omega u(\xi_1^2\xi_3^2 + 2\xi_1\xi_3\xi_4 + \xi_2\xi_3^2)dx =: J_6, \\
	0 &= \int_\Omega(uv_{xx}^3v_x)_x dx
	= \int_\Omega u(\xi_1^2\xi_2^3 + 3\xi_1\xi_2^2\xi_3 + \xi_2^4)dx =: J_7, \\
	0 &= \int_\Omega(uv_{xxx}v_{xx}^2)_x dx
	= \int_\Omega u(\xi_1\xi_2^2\xi_3 + 2\xi_2\xi_3^2 + \xi_2^2\xi_4)dx =: J_8.
\end{align*}
Then
\begin{align*}
  G''(0) &= 
	G''(0) - 4\sum_{i=1}^8 c_iJ_i = -\int_\Omega u\Big(a_1\xi_1^8 + a_2\xi_1^6\xi_2 
	+ a_3\xi_1^5\xi_3 + a_4\xi_1^4\xi_2^2
	+ a_5\xi_1^4\xi_4	 \\
	&\phantom{xx}{}	+ a_6\xi_1^3\xi_2\xi_3
	+ a_7\xi_1^2\xi_2^3 + a_8\xi_1^2\xi_2\xi_4 + a_9\xi_1^2\xi_3^2
	+ a_{10}\xi_1\xi_2^2\xi_3 + a_{11}\xi_1\xi_3\xi_4 + a_{12}\xi_2^4 \\
	&\phantom{xx}{}
	+ a_{13}\xi_2^2\xi_4 + a_{14}\xi_2\xi_3^2 + a_{15}\xi_4^2\Big)dx,
\end{align*}
where
\begin{align*}
  a_1 &= 4c_1, & a_2 &= 28c_1 + 4c_2, &	a_3 &= 4c_2 + 4c_3, \\
	a_4 &= 2 + 20c_2 + 4c_4, & a_5 &= 4c_3, & a_6 &= 8 + 16c_3 + 8c_4 + 4c_5, \\
	a_7 &= 5 + 12c_4 + 4c_7, & a_8 &= 4 + 4c_5, & a_9 &= 8 + 4c_5 + 4c_6, \\ 
	a_{10} &= 10 + 8c_5 + 12c_7 + 4c_8, &	a_{11} &= 8 + 8c_6, & a_{12} &= 3 + 4c_7, \\
	a_{13} &= 5 + 4c_8, & a_{14} &= 4c_6 + 8c_8, & a_{15} &= 2. & &  
\end{align*}
Next, we eliminate all terms involving $\xi_4$ by formulating the following square:
\begin{align*}
  & G''(0) = -\int_\Omega u\bigg[a_{15}
	\bigg(\xi_4 + \frac{a_5}{2a_{15}}\xi_1^4 + \frac{a_8}{2a_{15}}\xi_1^2\xi_2
	+ \frac{a_{11}}{2a_{15}}\xi_1\xi_3 + \frac{a_{13}}{2a_{15}}\xi_2^2\bigg)^2 \\
	&\phantom{xx}{}+ \bigg(a_1-\frac{a_5^2}{4a_{15}}\bigg)\xi_1^8
	+ \bigg(a_2-\frac{a_5a_8}{2a_{15}}\bigg)\xi_1^6\xi_2
	+ \bigg(a_3-\frac{a_5a_{11}}{2a_{15}}\bigg)\xi_1^5\xi_3 \\
	&\phantom{xx}{}+ \bigg(a_4-\frac{a_8^2}{4a_{15}}-\frac{a_5a_{13}}{2a_{15}}\bigg)
	\xi_1^4\xi_2^2 + \bigg(a_6-\frac{a_8a_{11}}{2a_{15}}\bigg)\xi_1^3\xi_2\xi_3
	+ \bigg(a_7-\frac{a_8a_{13}}{2a_{15}}\bigg)\xi_1^2\xi_2^3 \\
	&\phantom{xx}{}
	+ \bigg(a_9-\frac{a_{11}^2}{4a_{15}}\bigg)\xi_1^2\xi_3^2 
	+ \bigg(a_{10}-\frac{a_{11}a_{13}}{2a_{15}}\bigg)\xi_1\xi_2^2\xi_3
	+ \bigg(a_{12}-\frac{a_{13}^2}{4a_{15}}\bigg)\xi_2^4
	+ a_{14}\xi_2\xi_3^2\bigg]dx.
\end{align*}
We eliminate all terms involving $\xi_3$ and set the corresponding 
coefficients to zero. From $a_{14}=0$ we conclude that $c_6=-2c_8$. Furthermore,
\begin{align*}
  a_9-\frac{a_{11}^2}{4a_{15}} &= 0\ &\mbox{gives}\qquad
	c_5 &= 8c_8^2-6c_8, \\
	a_{10}-\frac{a_{11}a_{13}}{2a_{15}} &= 0\ &\mbox{gives}\qquad
	c_7 &= -\frac{20}{3}c_8^2 + \frac83 c_8, \\
	a_6-\frac{a_8a_{11}}{2a_{15}} &= 0\ &\mbox{gives}\qquad
	c_4 &= -2c_3 - 16c_8^3 + 16c_8^2 - 5c_8, \\
	a_3-\frac{a_5a_{11}}{2a_{15}} &= 0\ &\mbox{gives}\qquad
	c_2 &= c_3 - 4c_3c_8.
\end{align*}
By these choices, we obtain
$$
  b_{12} :=a_{12}-\frac{a_{11}^2}{4a_{15}} 
	= -\frac{86}{3}c_8^2 + \frac{17}{3}c_8 - \frac18.
$$
This quadratic polynomial in $c_8$ admits its maximal value at $c_8^*=17/172$
with value $b_{12}=20/129$.
The integral can now be written as
$$
  G''(0) \le -\int_\Omega u\big(b_1\xi_1^8 + b_2\xi_1^6\xi_2
	+ b_4\xi_1^4\xi_2^2 + b_7\xi_1^2\xi_2^3 + b_{12}\xi_2^4\big)dx, 
$$
where
\begin{align*}
  & b_1 = a_1 - \frac{a_5^2}{4a_{15}}
	= 4c_1 - 2c_3^2, \\
	& b_2 = a_2 - \frac{a_5a_8}{2a_{15}}
	= 28c_1 - 32c_3c_8^2 + 8c_3c_8, \\
	& b_4 = a_4 - \frac{a_8^2}{4a_{15}} - \frac{a_5a_{13}}{2a_{15}}
	= 7c_3 - 84c_3c_8 - 128c_8^4 + 128c_8^3 - 40c_8^2 + 4c_8, \\
	& b_7 = a_7 - \frac{a_8a_{13}}{2a_{15}}
	= -24c_3 - 244c_8^3 + \frac{448}{3}c_8^2 - \frac{70}{3}c_8.
\end{align*}
If $b_4=2b_2b_{12}/b_7 + b_7^2/(4b_{12})$, we can write the integal
as the sum of two squares, noting that $b_{12}$ is positive,
$$
  G''(0) \le -\int_\Omega u\bigg(b_{12}
	\bigg(\xi_2^2 + \frac{b_7}{2b_{12}}\xi_1^2\xi_2 + \frac{b_2}{b_7}\xi_1^4\bigg)^2
	+ \bigg(b_1 - \frac{b_2^2b_{12}}{b_7^2}\bigg)\xi_1^8\bigg)dx.
$$
The expression $b_4b_7-2b_2b_{12} - b_7^3/(4b_{12})=0$ defines a polynomial 
in $(c_1,c_3)$ which is linear in $c_1$. Solving it for $c_1$ gives
$$
  c_1 = \frac{449307}{175}c_3^3 + \frac{741681}{2150}c_3^2
	+ \frac{35780649411}{2393160700}c_3 + \frac{34135130165539}{163091166664200}.
$$
It remains to show that $p(c_3):=b_1 - b_2^2b_{12}/b_7^2$,
which is a polynomial of fourth order in $c_3$, is positive. 
Choosing $c_3^*=-0.029$, we find that $p(c_3^*)\approx 0.0045>0$. This shows that
$$
  G''(0) \le -q(c_3^*)\int_\Omega u\xi_1^8 dx
	= -q(c_3^*)\int_\Omega u(\log u)_x^8 dx \le 0.
$$
Finally, if $G''(0)=0$, we infer that $u$ is constant which is excluded.
Therefore, $G''(0)<0$, which ends the proof.
\end{proof}

%%%%%%%%%%%%%%%%%%%%%%%%%%%%%%%%%%%%%%%%%%%%%%%%%%%%%%%%%%%%%%%%%%%%%%%%%%%%%%

\section{Numerical examples}\label{sec.num}

The aim of this section is to explore the numerical behavior of the 
second-order derivative of the function $G(\tau)$, defined in the introduction,
for the porous-medium equation
\eqref{pm.eq} in one space dimension. The equation is discretized by standard
finite differences, and we employ periodic boundary conditions. The discrete
solution $u_i^k$ approximates the solution $u(x_i,t^k)$ to \eqref{pm.eq}
with $x_i=i\triangle x$, $t^k=k\tau$, and $\triangle x$, $\tau$ are the
space and time step sizes, respectively. We choose the Barenblatt profile
\begin{equation}\label{num.ic}
  u^0(x) = t_0^{-1/(\beta+1)}\max\bigg(0,C-\frac{\beta-1}{2\beta(\beta+1)}
	\frac{(x-1/2)^2}{t_0^{2/(\beta+1)}}\bigg)^{1/(\beta-1)}, \quad 0\le x\le 1,
\end{equation}
where 
$$
  t_0=0.01, \quad 
	C = \frac{\beta-1}{2\beta(\beta+1)}\frac{(x_R-1/2)^2}{t_0^{2/(\beta+1)}}, \quad
	x_R = \frac14,
$$
as the initial datum. Its support is contained in $[\frac12-x_R,\frac12+x_R]$;
see Figure \ref{fig.evol} (left). We choose the exponent $\beta=2$.
The semi-logarithmic plot of the discrete
entropy $H_d[u^k]=\sum_{i=0}^N (u_i^k)^\alpha \triangle x$ with $\alpha=5$
versus time is illustrated in
Figure \ref{fig.evol} (right), using the implicit Euler scheme with
parameters $\tau=10^{-4}$ and the number of grid points $N=1/\triangle x=64$. 
The decay is exponential for ``large'' times.
The nonlinear discrete system is solved by Newton's method with the tolerance
${\tt tol}=10^{-15}$. We have highlighted four time steps $t_i$ at which
we will compute numerically the function $G(\tau)$ 
for the following Runge-Kutta schemes:
\begin{align*}
  &\mbox{explicit Euler scheme:} & u^k-u^{k-1} &= -\tau A[u^{k-1}], \\
	&\mbox{implicit Euler scheme:} & u^k-u^{k-1} &= -\tau A[u^{k}], \\
	&\mbox{second-order trapezoidal rule:} & u^k-u^{k-1} 
	&= -\frac{\tau}{2}(A[u^k]+A[u^{k-1}]), \\
	&\mbox{third-order Simpson rule:} & u^k-u^{k-1} 
	&= -\frac{\tau}{6}(A[u^k]+4A[(u^k+u^{k-1})/2]+A[u^{k-1}]).
\end{align*}

\begin{figure}[ht]
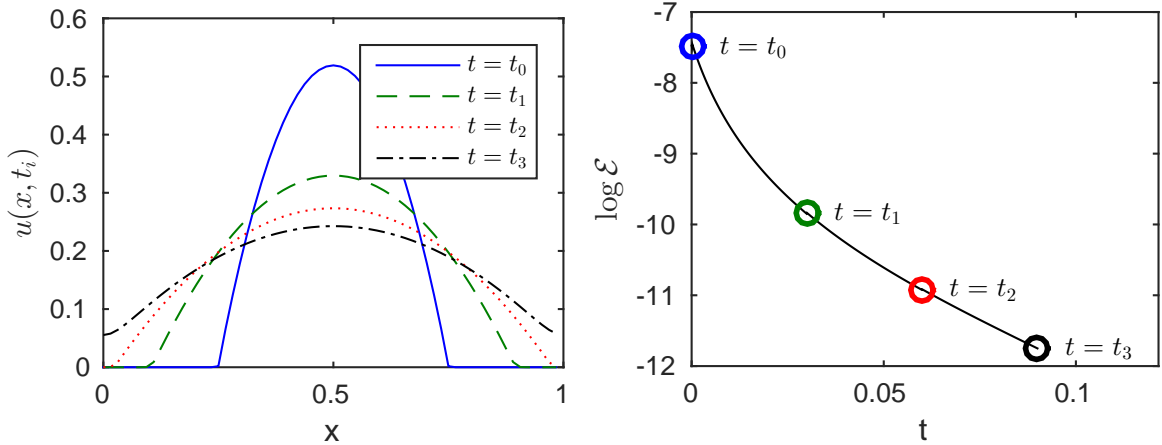

\includegraphics[width=75mm]{GSS_usnap_PME_3.eps}\hskip3mm
\includegraphics[width=75mm]{GSS_entropy_PME_3.eps}
\caption{Left: Evolution of the intial datum \eqref{num.ic} for $\beta=2$
at various time steps $t_i$, $i=0,1,2,3$. Right: Semi-logarithmic plot of
the discrete entropy $H_d[u^k]$ versus time.}
\label{fig.evol}
\end{figure}

We set as before $u:=u^k$, $v(\tau):=u^{k-1}$ and compute
$G(\tau)=H_d[u]-H_d[v(\tau)]$ and the discrete second-order derivative $\pa^2 G$
of $G$ (using central differences). The result is presented in Figure \ref{fig.G}.
As expected, the discrete derivative $\pa^2G$ is negative on a (small) interval
for all times $t_i$, $i=1,2,3$. 
We observe that $\pa^2 G$ is even slightly decreasing, but we expect that
it becomes positive for sufficiently large values of $\tau$.
Clearly, the values for $\pa^2G$ tend to zero as we approach the
steady state (see Remark \ref{rem.tauk}). This experiment indicates that
$\tau^k$ from Theorem \ref{thm.ed} is bounded from below by 
$\tau^*=3\cdot 10^{-4}$, for instance. 

\begin{figure}[ht]
\includegraphics[width=140mm]{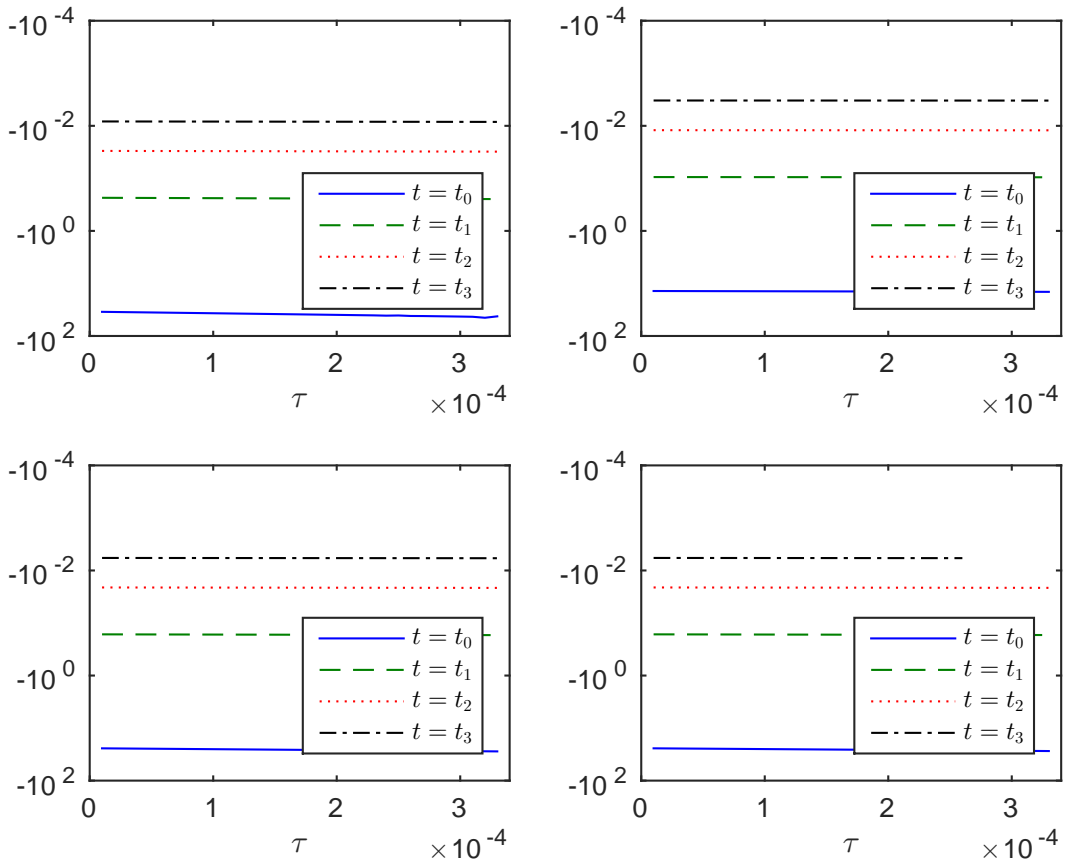}
\caption{Numerical evaluation of the discrete version of $G''(\tau)$
for various Runge-Kutta schemes at the time steps $t_i$. Top left: explicit
Euler scheme; top right: implicit Euler scheme; bottom left: implicit
trapezoidal rule; bottom right: Simpson rule.}
\label{fig.G}
\end{figure}

In order to understand the behavior of $G(\tau)$ in a better way, it is convenient
to study the discrete version of the quotient
\begin{equation}\label{num.Q}
  Q(\tau) := \frac{G''(\tau)}{\|u^{\alpha+2\beta-2}u_x^4\|_{L^1}}.
\end{equation}
Indeed, the analysis in Section \ref{sec.pme} gives an estimate of the type
$G''(0)\le -C\int_\Omega u^{2\beta+\alpha-5}u_x^4 dx$ for some constant $C>0$.
Thus, we expect that for sufficiently small $\tau>0$,
$Q(\tau)$ is bounded from above by some negative constant. 
This expectation is confirmed in Figure \ref{fig.quot}. In the examples, $Q(\tau)$
is a decreasing function of $\tau$, and $Q(0)$ is decreasing with increasing time.

\begin{figure}[ht]
\includegraphics[width=140mm]{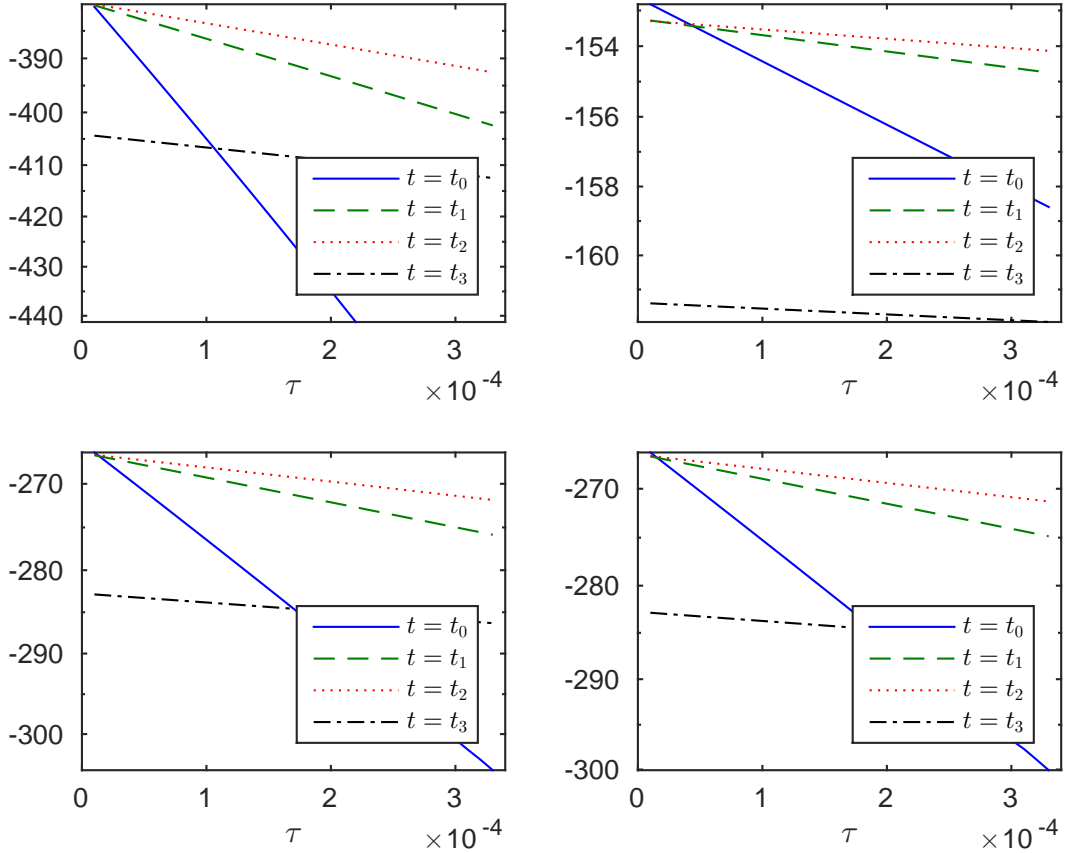}
\caption{Numerical evaluation of the discrete version of $Q(\tau)$,
defined in \eqref{num.Q}, 
for various Runge-Kutta schemes at the time steps $t_i$. Top left: explicit
Euler scheme; top right: implicit Euler scheme; bottom left: implicit
trapezoidal rule; bottom right: Simpson rule.}
\label{fig.quot}
\end{figure}

All these results indicate that
the threshold parameter $\tau^k$ in Theorem \ref{thm.ed} can be chosen
independently of the time step $k$.

%%%%%%%%%%%%%%%%%%%%%%%%%%%%%%%%%%%%%%%%%%%%%%%%%%%%%%%%%%%%%%%%%%%%%%%%%%%%%%


\begin{thebibliography}{11}
\bibitem{BaEm85} D.~Bakry and M.~Emery. Diffusions hypercontractives. In: 
{\em S\'eminaire de probabilit\'es XIX}, 1983/84, {\em Lect. Notes Math.} 
1123 (1985), 177-206.

\bibitem{Bes12} M.~Bessemoulin-Chatard. A finite volume scheme for 
convection-diffusion equations with nonlinear diffusion derived from the 
Scharfetter-Gummel scheme. {\em Numer. Math.} 121 (2012), 637-670.

\bibitem{BFR15} S.~Boscarino, F.~Filbet, and G.~Russo. High order semi-implicit 
schemes for time dependent partial differential equations. Preprint, 2015.
{\tt https://hal.archives-ouvertes.fr/hal-00983924}.

\bibitem{BEJ14} M.~Bukal, E.~Emmrich, and A.~J\"ungel. Entropy-stable and 
entropy-dissipative approximations of a fourth-order quantum diffusion equation. 
{\em Numer. Math.} 127 (2014), 365-396.

\bibitem{CaGu15} C.~Canc\`es and C.~Guichard. Numerical analysis of a robust 
entropy-diminishing finite-volume scheme for parabolic equations with gradient 
structure. Preprint, 2015. {\tt arXiv:1503.05649.}

\bibitem{CJS15} C.~Chainais-Hillairet, A.~J\"ungel, and S.~Schuchnigg. 
Entropy-dissipative discretization of nonlinear diffusion equations and 
discrete Beckner inequalities. To appear in 
{\em ESAIM Math. Model. Numer. Anal.}, 2015. {\tt arXiv:1303.3791.}

\bibitem{CMO11} S.~Christiansen, H.~Munthe-Kaas, and B.~Owren. Topics in 
structure-preserving discretization. {\em Acta Numerica} 20 (2011), 1-119.

\bibitem{Dei85} K.~Deimling. {\em Nonlinear Functional Analysis.} Springer, Berlin, 
1985.  

\bibitem{DLSS91} B.~Derrida, J.~Lebowitz, E.~Speer, and H.~Spohn. Fluctuations of a
stationary nonequilibrium interface. {\em Phys. Rev. Lett.} 67 (1991), 165-168.

\bibitem{Fil08} F.~Filbet. An asymptotically stable scheme for diffusive 
coagulation-fragmentation models. {\em Commun. Math. Sci.} 6 (2008), 257-280. 

\bibitem{GlGa09} A.~Glitzky and K.~G\"artner. Energy estimates for continuous and 
discretized electro-reaction-diffusion systems. {\em Nonlin. Anal.} 70 (2009), 
788-805.

\bibitem{HNW93} E.~Hairer, S.P.~N{\o}rsett, and G.~Wanner. {\em Solving Ordinary
Differential Equations I.} Springer, Berlin, 1993.

\bibitem{Jue09} A.~J\"ungel. {\em Transport Equations for Semiuconductors}.
Lect. Notes Phys.\ 773, Springer, Berlin, 2009.

\bibitem{JuMa06} A.~J\"ungel and D.~Matthes. An algorithmic construction of entropies
in higher-order nonlinear PDEs. {\em Nonlinearity} 19 (2006), 633-659.

\bibitem{JuMa08} A.~J\"ungel and D.~Matthes. The Derrida-Lebowitz-Speer-Spohn 
equation: existence, non-uniqueness, and decay rates of the solutions.
{\em SIAM J. Math. Anal.} 39 (2008), 1996-2015.

\bibitem{JuMi09} A.~J\"ungel and J.-P.~Mili\v{s}i\'{c}. A sixth-order nonlinear
parabolic equation for quantum systems. {\em SIAM J. Math. Anal.} 41 (2009), 
1472-1490.

\bibitem{JuMi15} A.~J\"ungel and J.-P.~Mili\v{s}i\'{c}. Entropy dissipative 
one-leg multistep time approximations of nonlinear diffusive equations.
To appear in {\em Numer. Meth. Part. Diff. Eqs.}, 2015. {\tt arXiv:1311.7540.}

\bibitem{LiMi13} M.~Liero and A.~Mielke. Gradient structures and geodesic convexity 
for reaction-diffusion systems. {\em Phil. Trans. Royal Soc. A.} 371 (2013), 20120346 
(28 pages), 2013. 

\bibitem{LiYu14} H.~Liu and H.~Yu. Entropy/energy stable schemes for evolutionary 
dispersal models. {\em J. Comput. Phys.} 256 (2014), 656-677.

\bibitem{Tad87} E.~Tadmor. Numerical viscosity of entropy stable schemes for systems 
of conservation laws I. {\em Math. Comp.} 49 (1987), 91-103.

\end{thebibliography}
\end{document}